\documentclass[a4paper,12pt]{amsart}
\usepackage[usenames,dvipsnames]{xcolor}
\usepackage{hyperref, tikz-cd}
\hypersetup{colorlinks=true,linkcolor={Brown},citecolor={Brown},urlcolor={Brown}}

\usepackage{microtype}
\usepackage{amssymb,amsmath,amsfonts,amsthm,bm,latexsym,amscd,verbatim,url,nicefrac,stmaryrd,enumerate,colonequals,dsfont,appendix,stackrel,mathrsfs,cleveref}
\crefname{exm}{Example}{Examples}
\crefname{cor}{Corollary}{Corollaries}
\crefname{prop}{Proposition}{Propositions}
\crefname{rmk}{Remark}{Remarks}
\crefname{lem}{Lemma}{Lemmata}

\usepackage[color,matrix,arrow]{xy}
\usepackage{float}
\usepackage{times}
\usepackage{graphicx}
\usepackage{fullpage,adjustbox,soul}
\usepackage{array}

\makeatletter
\newcommand{\oset}[3][0ex]{%
	\mathrel{\mathop{#3}\limits^{
			\vbox to#1{\kern-2\ex@
				\hbox{$\scriptstyle#2$}\vss}}}}
\makeatother

\RequirePackage{xcolor}

\numberwithin{equation}{section}%
\newtheorem{thm}[equation]{Theorem}
\newtheorem{prop}[equation]{Proposition}

\newtheorem{cor}[equation]{Corollary}

\theoremstyle{definition}
\newtheorem{dfn}[equation]{Definition}
\newtheorem{notn}[equation]{Notation}
\newtheorem{rmk}[equation]{Remark}
\newtheorem{exm}[equation]{Example}

\newtheorem{cons}[equation]{Construction}

\usepackage{bbold,etoolbox}
\newcommand{\one}{\mathbb{1}}

\newcommand{\A}{\mathbb{A}}
\newcommand{\B}{\mathbb{B}}
\newcommand{\C}{\mathbb{C}}

\newcommand{\G}{\mathbb{G}}

\newcommand{\N}{\mathbb{N}}
\renewcommand{\P}{\mathbb{P}}

\newcommand{\Q}{\mathbb{Q}}

\newcommand{\Z}{\mathbb{Z}}

\newcommand{\catC}{\mathscr{C}}

\newcommand{\catD}{\mathscr{D}}

\newcommand{\mcA}{\mathcal{A}}

\newcommand{\mcD}{\mathcal{D}}

\newcommand{\mcK}{\mathcal{K}}

\newcommand{\mcO}{\mathcal{O}}

\newcommand{\mcS}{\mathcal{S}}

\newcommand{\mcX}{\mathcal{X}}

\newcommand{\mfT}{\mathfrak{T}}

\newcommand{\mfP}{\mathfrak{P}}

\newcommand{\mfS}{\mathfrak{S}}
\newcommand{\mfU}{\mathfrak{U}}
\newcommand{\mfX}{\mathfrak{X}}

\newcommand{\xto}[1]{\xrightarrow{#1}}

\DeclareMathOperator{\alg}{alg}

\DeclareMathOperator{\an}{an}

\DeclareMathOperator{\An}{An}

\DeclareMathOperator{\CAlg}{CAlg}

\DeclareMathOperator{\colim}{colim}

\DeclareMathOperator{\dR}{dR}

\DeclareMathOperator{\eff}{eff}

\DeclareMathOperator{\Ext}{Ext}
\DeclareMathOperator{\et}{\acute{e}t}

\DeclareMathOperator{\Fr}{Fr}

\DeclareMathOperator{\FDA}{{FDA}}

\DeclareMathOperator{\Frac}{Frac}

\DeclareMathOperator{\FSm}{FSm}
\DeclareMathOperator{\Gal}{Gal}

\DeclareMathOperator{\gr}{gr}
\DeclareMathOperator{\Hm}{H}

\DeclareMathOperator{\Hom}{Hom}

\DeclareMathOperator{\id}{id}

\DeclareMathOperator{\Ind}{Ind}

\DeclareMathOperator{\Map}{Map}
\DeclareMathOperator{\map}{map}
\DeclareMathOperator{\MHS}{MHS}

\DeclareMathOperator{\op}{{op}}

\DeclareMathOperator{\Perf}{Perf}

\newcommand{\Prl}{{\rm{Pr}^{L}}}

\newcommand{\Prlm}{{\rm{CAlg}(\rm{Pr}^{L})}}

\newcommand{\Prlo}{{\rm{Pr}^{L}_\omega}}

\newcommand{\Prloo}{{\rm{CAlg}(\Prlo)}}

\DeclareMathOperator{\proet}{pro\acute{e}t}

\DeclareMathOperator{\qcqs}{qcqs}

\DeclareMathOperator{\rig}{rig}

\DeclareMathOperator{\RigSm}{RigSm}

\DeclareMathOperator{\Sm}{Sm}

\DeclareMathOperator{\Spa}{Spa}
\DeclareMathOperator{\Spec}{Spec}

\DeclareMathOperator{\Spf}{Spf}

\DeclareMathOperator{\uhom}{\underline{Hom}}

\DeclareMathOperator{\DA}{{{DA}}}
\DeclareMathOperator{\DAet}{{{DA}_{\et}}}

\DeclareMathOperator{\Mod}{{Mod}}

\DeclareMathOperator{\QCoh}{{QCoh}}

\DeclareMathOperator{\RigDA}{{RigDA}}
\DeclareMathOperator{\RigDAeff}{{RigDA^{\eff}}}
\DeclareMathOperator{\RigDM}{{RigDM}}

\DeclareMathOperator{\Sch}{{Sch}}

\DeclareMathOperator{\SH}{{SH}}

\DeclareMathOperator{\pt}{{pt}}

\DeclareMathOperator{\fib}{fib}
\DeclareMathOperator{\cof}{cofib}

\usepackage{xspace}
\newcommand{\icat}{$\infty$\nobreakdash-category\xspace}
\newcommand{\icats}{$\infty$\nobreakdash-categories\xspace}

\newcommand{\subicat}{sub-$\infty$-category\xspace}

\usepackage{array}

\makeatletter 
\newcount\SOUL@minus
\makeatother  

\begin{document}
	\title{A motivic approach to rational $p$-adic cohomologies}
	\author{Federico Binda}
 \address{Dipartimento di Matematica ``F. Enriques'' - Universit\`a degli Studi di Milano (Italy)}
 \email{federico.binda@unimi.it}
	
	\author{Alberto Vezzani}
  \address{Dipartimento di Matematica ``F. Enriques'' - Universit\`a degli Studi di Milano (Italy)}
  \email{alberto.vezzani@unimi.it}

	\thanks{
		The authors are partially supported by the PRIN 20222B24AY ``The arithmetic of motives and $L$-functions''}
	
	\begin{abstract}
We survey over some recent applications of motivic homotopy theory in the definition and the study of $p$-adic cohomology theories. In particular, we revisit the proof of the $p$-adic weight-monodromy conjecture for smooth projective hypersurfaces in light of the motivic definition of nearby cycles and monodromy operators.
	\end{abstract}

	\maketitle
 	\setcounter{tocdepth}{1}
 	\tableofcontents
	
 \section{Introduction}

\subsection{From the special fiber to the generic fiber}
One of the first motivations for the introduction of analytic geometry over non-archimedean fields is the attempt of defining a $p$-adic Weil cohomology  theory (as opposed to the $\ell$-adic \'etale cohomology with $\ell\neq p$) for varieties over a finite field $k$ of characterstic $p$, and the study of the associated zeta functions. 

Following ideas of Dwork, Monsky and Washnitzer \cite{mw-fc1} introduced a way to attach to an affine, smooth variety over $k$ a vector space $H^i_{MW}(X/K)$ over $K=\Frac W(k)$. This can be thought as an ``overconvergent'' de Rham cohomology of a (dagger rigid analytic) generic fiber of a weakly complete smooth  lift of $X$ over $W(k)$. 

We may interpret Monsky-Washnitzer's plan was as a way to define a ``functor''
\begin{equation}\label{squiggly}
 \Sm/k \leadsto \FSm/\mcO_K \leadsto \RigSm/K \leadsto \RigSm^\dagger/K    
\end{equation}
with $K$ being the $p$-adic non-archimedean field $W(k)[1/p]$, and where $\Sm$ [resp. $\FSm$ resp. $\RigSm$ resp. $\RigSm^\dagger$] stands for the category of smooth varieties [resp.  formal schemes, resp.  rigid analytic variety, resp. dagger varieties] over the the base.

Nonetheless, of all the arrows sketched above, only the middle one is an actual functor. As a matter of fact, the existence of some formal lift is only granted locally, for affine varieties $X$. It is not even enough to work locally (that is, considering the Zariski topos): indeed, even if affine smooth lifts are  unique up to isomorphism,  their ``weak completions'' depend on a choice of coordinates for $X$. The same ambiguity exists when trying to lift maps between affine smooth varieties.

It was only thanks to Berthelot's introduction of the overconvergent site \cite{berth-pre} that   Monsky--Washnitzer cohomology could be turned into a well-defined functorial cohomology theory, extendable to arbitrary varieties over $k$ and with the usual properties of a Weil cohomology theory (a K\"unneth formula, finite dimensionality etc.).

We propose an alternative way to ``straighten up'' the squiggly arrows of \eqref{squiggly} based on motivic homotopy theory. Indeed, by results of Elkik \cite{elkik} and Arabia \cite{arabia}, we know that potential lifts of maps are unique,  \emph{up to homotopy}. Such results suggest that one should consider a more flexible environment, in which not only can one work Zariski-locally (or even \'etale locally) but also \emph{up to homotopy}. More or less by definition, this environment is the so-called category of (derived, \'etale, rational) motives that we will denote by $\DA$ or $\RigDA$ or $\RigDA^\dagger$  in the algebraic resp. analytic  resp. dagger setting.

The aim of \Cref{sec:fromktoK}, in which we  collect results from \cite{vezz-MW,lbv,ev} is to sketch how one can prove the following results.

\begin{prop}
    There is a functor
$$
\xi\colon\Sm/k\to \DA(k)\to \RigDA^\dagger(K)
$$
Whenever $X/k$ admits a weak formal smooth lift $\mfX^\dagger$ over $W(k)$, then $\xi(X)\simeq \mathsf{M}(\mfX^\dagger_K)$.
\end{prop}

We can use this functor to define Monsky--Washnitzer/rigid cohomology and prove even new features of them. In what follows, we let $\Fr/S$ be the category of proper frames over $S$ in the sense of Berthelot.

\begin{thm}\label{intro:rig}
    Let $S$ be a scheme over $k$. There is a functor
    $$
\xi\colon\Sm/S\to\DA(S)\to\lim_{(T,\bar{T},\mfP)\in\Fr/S}\RigDA^\dagger(]{T}[^\dagger_{\mfP})
    $$
   which, composed to an overconvergent relative de Rham functor gives rise to a cohomology theory: 
    $$
    R\Gamma_{\rig}\colon  \colon\Sm/S\to\DA(S)\to\lim_{(T,\bar{T},\mfP)\in\Fr/S}\QCoh(]{T}[^\dagger_{\mfP})^{\op}.
    $$
    computing (relative) rigid cohomology $R\Gamma_{\rig}(X/(T,\bar{T},\mfP))$.
    
    In particular, if $S=\Spec k$ we get an absolute Monsky--Washnitzer realization:
    $$
    R\Gamma_{\rig}\colon\Sm/k\to\mcD(K)^{\op}.
    $$
\end{thm}

\begin{cor}
    Let $S$ be a scheme over $k$.
    \begin{enumerate}
    \item \label{berth}If $S=\Spec k$, then (absolute) Monsky--Washnitzer cohomology is well-defined, functorial, has \'etale and h-descent, and can be extended to arbitrary varieties by putting $R\Gamma_{\rig}(X/K)=R\Gamma_{\rig}(\mathsf{M}_k(X))$. It has the structure of a $\varphi$-module, and is finite dimensional whenever $X$ is a smooth quasi-compact rigid variety, or the analytification of an algebraic variety.
        \item \label{bconj}Relative rigid cohomology has \'etale descent and h-descent. If $X/S$ is smooth and proper, then $R^i\Gamma_{\rig}(X/(T,\bar{T},\mfP))$ is a vector bundle over $]T[^\dagger_{\mfP}$ and is equipped with a canonical structure of an overconvergent $F$-isocrystal. 
    \end{enumerate}
\end{cor}
We note that point \eqref{berth} was essentially established by Berthelot and Chiarellotto--Tsuzuki, so that motivic methods help re-visit classical resuls. On the other hand, point \eqref{bconj} answers positively to an open conjecture by Berthelot which was previously known only in the (log-)liftable case, and for a special class of frames \cite{berth-pre,tsuzuki_coherence,shiho_relative3}.

\subsection{From generic fibers to special fibers}
One of the classical ways to define an analytic variety over a non-archimedean field arises when one has a generically smooth family $X_t$ of varieties over a curve $C$ over a field $k$ of characteristic $0$. When restricting around an infinitesimal ``punctured neighborhood'' of a point $0\in C$,  one ends up with a formal model of a smooth rigid analytic variety over $k(\!(t)\!)$, having a (potentially singular) special fiber $X_0$ over $k$. By standard resolution of singularities, one may even suppose that such special fiber is semi-stable (at least after an extension of scalars).

By classic work by Clemens \cite{clemens}, Schmid \cite{schmid}, Steenbrink \cite{steenbrink}, Saito \cite{saito-mdhp}  et al.~ when $k=\C$, the limit Hodge cohomology of the elements of the family $\lim_{t\to 0}H(X_t)$ can be expressed in terms of a nearby cycle functor $\Psi X$, introduced either via passing to a universal covering and using the functor $\iota^*j_*$, or as an explicit bi-complex of Hodge modules depending on the geometry of $X_0$. It comes equipped with a monodromy operator.

In the arithmetic situation, by replacing the punctured complex disc and the non-archimedean field $\C(\!(t)\!)$ with   a local field $K$, this last procedure has been translated to the study of $\ell$-adic cohomology by Rapoport--Zink \cite{Rapoport_Zink}, interpreting the Weil-Deligne cohomology of the generic fiber in terms of $\ell$-adic cohomology of a bi-complex $\Psi X$, equipped with a monodromy operator. Similarly, this procedure was adapted to the $p$-adic case, with the introduction of a so called Hyodo--Kato cohomology,  by work of Hyodo--Kato \cite{HK94}, Mokrane \cite{mokrane}, Nakkajima \cite{Nakka}, Gro\ss e-Kl\"onne \cite{GK05}, Colmez--Niziol \cite{CN19} and Ertl--Yamada \cite{EY19}. Once again, Hyodo--Kato cohomology of the (rigid analytic) generic fiber can be described as rigid cohomology of a bi-complex $\Psi X$, and comes equipped with a monodromy operator. 

In all three circumstances, the strategy is then to consider a ``nearby cycle functor''
$$
\Psi\colon \RigSm/K\leadsto \Sm/k + \text{monodromy}
$$
and then define/study a ``limit'' cohomology for $X/K$ via a cohomology theory for $\Psi X$. Once again, this functor is not really well-defined, as typically the  nearby cycle functor is given by some bi-complex of objects in the category of \emph{coefficients} (Hodge modules, $\ell$-adic sheaves or $\varphi$-modules) albeit defined using the geometry of the special fiber.

Also in this setting, working with a motivic category allows one to straighten up this arrow again into a well-defined functor, that we will discuss in \Cref{sec:fromKtok} based on results of \cite{bgv, pWM}:
\begin{thm}\label{thm:intro-MON}
    Let $C$ be an algebraically closed non-archimedean field, with valuation group equal to $\Q$ and residue field $k$. There is an equivalence of categories
    $$
    \RigDA(C)\simeq \DA(k)_N
    $$
    where the category on the right hand side is the category of locally nilpotent  ``monodromy maps'' $M\xto{N} M(-1)$ with $M$  in $\DA(k)$. 

    In particular, if $X$ is a quasi-compact rigid variety or an algebraic variety over a local field $K$, then up to replacing $K$ with a finite extension, $X$ is canonically associated to a motive $\Psi X\in\DA(k)$ equipped with a monodromy operator.
\end{thm}

This provides a geometric (yet, motivic) universal way to perceive and define the various functors $\Psi$ on each realization, and even prove new facts about the associated ``limit cohomologies''.

\begin{cor}
    Hyodo--Kato cohomology for rigid analytic varieties over $C$  can be defined as $R\Gamma_{\rig}(\Psi X)$ without recourse to log-geometry. It is canonically equipped with a weight filtration and fits in a Clemens-Schmid chain complex.
\end{cor}

Once again, while the first point of the previous corollary admits some variations which were already known, the second point is new, and answers positively to a conjecture by Flach--Morin \cite{flach-morin}.

\subsection{On the $p$-adic weight-monodromy conjecture}
As a matter of fact, this ``motivic'' intepretation of the Hyodo--Kato cohomology also sheds some new light on the   proof of the $p$-adic weight monodromy conjecture for smooth projective hypersurfaces, obtained in  \cite{pWM} based on Scholze's strategy \cite{scholze}. The most precise statement goes as follows:

\begin{thm}[{\cite{pWM}}]
    Let $K$ be a finite extension of $\Q_p$ and $Y$ be a smooth scheme-theoretic complete intersection inside a projective smooth toric variety over $K$. Then the p-adic weight-monodromy conjecture holds for $Y$.
\end{thm}

Indeed, one of the technical ingredients of the main proof of \cite{pWM}, is the compatibility of the motivic tilting equivalence $\RigDA(C)\simeq\RigDA(C^\flat)$ from \cite{vezz-fw}  with the Hyodo--Kato realization functors based on log-geometry. We will show that this check can alternatively be done based on the identification of \Cref{thm:intro-MON} and some detection principles developed in \cite{bgv}.

We will  revisit the strategy of the above theorem in \Cref{sec:wm}, by making some generalizations of certain intermediate results in light of the equivalence of \Cref{thm:intro-MON}. As such, this section contains some original results (see the proof of \Cref{original}).

\subsection{Acknowledgements}
We thank the organizers and the participants of the conference ``Regulators V'' for their hospitality and the useful interactions, and an anonymous referee for their comments and  corrections to an early version of this paper.

\section{Properties of motivic categories}
In this section, we list all the ``formal'' properties of motivic categories that we use for our explicit, cohomological applications later on. We will survey them based on \cite{ayoub-weil, agv, lbv, bgv}.
 \subsection{Basic definitions}
Motivic sheaves are defined by means of a universal property, that we now recall.

 \begin{dfn}\label{def:mot_sheaves} We will consider the following \icats of motivic sheaves.\label{dfn:motives}
			\begin{enumerate}
				\item \label{eq:algebraic_motives}
				Let $S$ be a scheme locally of finite Krull dimension. We write $\DA(S) = \DAet(S;\Q)$ for the \icat of hypercomplete \'etale motives over~$S$, with rational coefficients as defined in~\cite[Section 3]{AyoubEt} by putting~$\Lambda=\Q$. It is constructed from the full \subicat $\DA^{\eff}_{\et}(S;\Q)$ of $\mathrm{Shv}_{\et}(\Sm/S; \mcD(\Q))$  spanned by those objects which are local with respect to the collection of maps $\Q(\A^1_Y) \to \Q(Y)$, by formally inverting the tensor product operation with the Tate twist $\one(1)$.  
				\item\label{eq:analytic_motives} Let $S$ be a rigid analytic space locally of finite Krull dimension, in the sense of \cite[Definition II.2.2.18]{fujiwara-kato}. We write $\RigDA(S) = \RigDA_{\et}(S; \Q)$ for the \icat of hypercomplete rigid analytic motives over~$S$ with respect  to the étale topology, as defined in~\cite[Definition 2.1.15]{agv}. Its construction is analogous to~\eqref{eq:algebraic_motives}, by replacing~$\Sm/S$ with~$\RigSm/S$, and the affine line~$\A^1$ by the closed unit disk~$\B^1$.  %
			\end{enumerate}
			
			These are stable, compactly generated monoidal \icats, underlying a six-functor formalism, see~\cite{ayoub-th1,ayoub-th2,agv}.
		\end{dfn}
		\begin{rmk}
			In the setting of \eqref{eq:analytic_motives}, we will mostly be interested in the case where $S= \Spa(K,\mcO_K)=\Spf(\mcO_K)^{\rig}$, for~$K$ a complete non-archimedean field with a perfect residue field. Here, for a  quasi-compact and quasi-separated
formal scheme $\mfX$, we write $\mfX^{\rig}$ or $\mfX_\eta$ for the Raynaud generic fiber. We will simply write $\RigDA(K)$ for $\RigDA(\Spf({\mcO_K)^{\rig}})$. Note that in this case by \cite[Lemma 2.4.18]{agv} the non-hypercomplete version of the construction of $\RigDA$ agrees with the one considered above, and the reader can simply ignore the difference between the two.
		\end{rmk}
\subsection{Weil cohomology theories in algebraic and analytic geometry} We follow the exposition of Ayoub \cite[\S 1]{ayoub-weil}, and quickly recall the notion of mixed  Weil cohomology theory, as introduced by Cisinski and Déglise \cite[\S 2]{CD12},  specialized to our setting of $\Q$-linear stable $\infty$-categories.

We begin by recalling the following basic definition.
\begin{dfn}\label{def:realization} Let $k$ be a field (resp.~let $K$ be a complete non-archimedean local field).  A monoidal \emph{realization functor} or simply a realization is a morphism
    \[R\Gamma^*\colon \DA(k) \to \catD^\otimes \]
(resp., $R\Gamma^*\colon \RigDA(K) \to \catD^\otimes$) in $\Prlm$, where $\catD^\otimes$ is presentable monoidal stable \icat. We will mostly be interested in the case where $\catD^\otimes\in \Mod_{\Mod_\Q}(\Prl)$ is $\Q$-linear. A ``classical'' example is given by $\catD^\otimes = \Mod_\Lambda$, for
$\Lambda\in \CAlg(\Mod_\Q)^{\rm cn}$ (i.e., a commutative connective dg $\Q$-algebra).
\end{dfn}
\begin{rmk} To give a realization functor for $\DA(k)$ is thus equivalent to give a symmetric monoidal functor $\Sm/k \to  \catD^\otimes$ satisfying étale descent, $\A^1$-invariance and Tate stability. Similarly, to give a realization functor for $\RigDA(K)$ is equivalent to give a symmetric monoidal functor $\RigSm/K\to  \catD^\otimes$ satisfying rig-étale descent, $\B^1$-invariance and Tate stability.  
\end{rmk}

Realizations that are $\Mod_\Lambda$-valued   are classically introduced in terms of cohomology theories for algebraic (resp.~analytic) varieties in the following sense.
\begin{dfn}[{\cite[Definition 1.1 and 2.1]{ayoub-weil}}] \label{df:mixed_Weil}Let $k$ be a field (resp.~let $K$ be a complete non-archimedean local field), and let $\Lambda\in \CAlg(\Mod_\Q)^{\rm cn}$.

A \emph{mixed Weil cohomology theory} for algebraic $k$-varieties (resp.~for rigid analytic $K$-varieties) is a presheaf $\Gamma_W$ of commutative $\Lambda$-algebras on $\Sm/k$ (resp.~on $\RigSm/K$) satisfying the following properties: 
\begin{enumerate}
    \item For any $X\in \Sm/k$ (resp.~$X\in \RigSm/K$) the natural morphism $\Gamma_W(X) \to \Gamma_W(X\times \A^1)$ (resp.~$\Gamma_W(X)\to \Gamma_W(X\times \B^1)$) is an equivalence;
    \item The $\Gamma_W(\pt)$-module $\cof(\Gamma_W(\pt)\to \Gamma_W(\P^1))$ is invertible;
    \item For every $X, Y\in \Sm/k$ (resp.~for every $X,Y\in \RigSm/K^{\qcqs}$) there is an equivalence
    \[ \Gamma_W(X) \otimes_{\Gamma_W(\pt)} \Gamma_W(Y) \simeq \Gamma_W(X\times Y);\]
    \item $\Gamma_W(-)$ has étale descent. 
\end{enumerate}
Note that $\Gamma_W$ is automatically an object of $\DA^{\rm eff}(k;\Lambda)$ (resp.~of $\RigDAeff(K;\Lambda)$).
\end{dfn}
\begin{rmk} Any mixed Weil cohomology theory for algebraic or for analytic varieties gives rise naturally to a commutative algebra object $\mathbf{\Gamma}_W$ in $\DA(k;\Lambda)$ (resp.~in $\RigDA(K;\Lambda)$), satisfying some extra assumptions. See \cite[Proposition 1.9, 2.5]{ayoub-weil} and \cite[2.1.5]{CD12}. In fact, the categories of \emph{Weil spectra} and that of $\Lambda$-valued Weil cohomology theories \cite[Definition 1.2]{ayoub-weil} are equivalent. See \cite[Theorems 1.16, 2.12]{ayoub-weil}
\end{rmk}
\begin{rmk}
Given any monoidal realization functor $R\Gamma^*$ as in \Cref{def:realization}, write $R\Gamma_*$ for its right adjoint. We denote $R\Gamma_*(\one)$ by $E_\Gamma$. Since $R\Gamma^*$ is monoidal, $R\Gamma_*$ is lax monoidal. In particular, it carries commutative algebra objects to commutative algebra objects, so that $E_\Gamma \in \CAlg (\DA(k))$ (resp.~in $\CAlg(\RigDA(K))$).      

The assignment $R\Gamma^* \mapsto  
E_\Gamma$ determines in fact an equivalence of categories between the category of algebraic (resp. analytic) $\Mod_\Lambda$-valued realization functors and the category of mixed Weil spectra in $\DA(k)$ (resp.~in $\RigDA(K)$). See \cite[Theorems 1.21, 2.16]{ayoub-weil}. 
\end{rmk}

We list a series of examples of classical realization functors. In what follows, we keep the convention that $k$ is a field and $K$ is a complete non-archimedean local field. 
\begin{exm}[$\ell$-adic realizations] Let $\ell$ be a prime invertible in $k$ (resp.~in $K$). The classical $\ell$-adic \'etale cohomology is an example of mixed Weil cohomology theory: for every $X\in \Sm/k$ (resp.~$X\in \RigSm/K$), write $R_\ell(X) \in \mathcal{D}_{\proet}(k, \Q_\ell)$ for the object computing the $\ell$-adic \'etale homology of $X$, see e.g.~\cite[Théorème 6.9]{AyoubEt} in the algebraic context. The assignment $X\mapsto R_\ell(X)$ satisfies all the assumptions of \Cref{df:mixed_Weil} and gives rise to a monoidal functor $\DA(k) \to \mathcal{D}_{\proet}(k, \Q_\ell)$ (resp.~$\RigDA(K) \to \mathcal{D}_{\proet}(k, \Q_\ell)$). 

One can  describe more accurately the $\ell$-adic \'etale realization (both in the algebraic and in the analytic setting), and over a general base, by means of the so-called rigidity theorem: in this language, the étale realization arises as the left adjoint to the canonical functor induced by the inclusion of the small étale site into the big étale site. See \cite[\S 2.10]{agv} and \cite[Theorem 3.2]{bamb-vezz} in the analytic setting and the already mentioned \cite{AyoubEt} in the algebraic setting for the compatibility of the $\ell$-adic realization with the six-functor formalism. 
\end{exm}
\begin{exm}[Betti and Hodge realization] \label{ex:Betti}Let $k=\C$. For a finite type $\C$-scheme $X$, let $X^{\rm an}$ be the set of $\C$-points $X(\C)$, equipped with the classical complex analytic topology. The functor $X\mapsto X^{\rm an}$ can be promoted to a symmetric monoidal functor $\mathsf{Bti} \colon \DA(\C) \to \DA^{\rm an}(\C) \simeq\Mod_\Q$, where $\DA^{\rm an}(\C)$ is the (stabilization of the) category of $\mathbb{D}^1_\C$-local sheaves on $\C$-smooth analytic spaces. By \cite[\S 2]{ayoub-h1} it is equivalent to the classical derived $\infty$-category of $\Q$-modules. For any smooth $\C$-scheme $X$, the homology groups $H_*(\mathsf{Bti}(\mathsf{M}(X)))$ compute the singular homology of $X(\C)$. 

The Betti realization can be further enriched to obtain a mixed Hodge realization. Let $\MHS^p_\Q$ be the Abelian category of polarisable mixed Hodge structures and let $\mathcal{D}^b(\MHS^p_\Q)$ be its bounded derived stable $\infty$-category.  The presheaf defined by sending $X\in \Sm/\C$ to the object $R\Gamma_{\rm Hodge}^*(X)\in \mathcal{D}^b(\MHS^p_\Q)$  computing the singular cohomology of $X$ equipped with its natural mixed Hodge structure is a mixed Weil cohomology theory. As such, it can be promoted to a symmetric monoidal (cohomological) realization functor
\[ \DA(\C) \to \mathcal{D}^b(\MHS^p_\Q)^{\rm op}.\]
Restricting to the subcategory of compact objects $\DA(\C)_\omega$ and pre-composing  with the duality functor $M\mapsto M^\vee$ gives the homological Hodge realization functor $\DA(\C)_\omega \to \mathcal{D}^b(\MHS^p_\Q)$. Since $\DA(\C) \simeq \Ind (\DA(\C)_\omega)$, this can be further extended to a symmetric monoidal functor $\DA(\C) \to \Ind \mathcal{D}^b(\MHS^p_\Q)$ taking values in the Ind-derived \icat. As for the $\ell$-adic realization, it turns out that also the Hodge realization can be defined in the relative setting, in a way that is compatible with the six-functor formalism. See \cite[Theorem 4.4]{Tubach}.
\end{exm}
\begin{exm}[de Rham realization] Let $k$ be a field of characteristic zero. 
    The presheaf $X\mapsto R\Gamma(X, \Omega_{X/k})$ is another example of mixed Weil cohomology, giving rise to a monoidal (cohomological) realization functor $\dR\colon\DA(k) \to \mathcal{D}(k)^{\rm op}$. We denote by $R\Gamma_{{\rm dR}, k}$ the covariant realization functor obtain by pre-composing the restriction of $\dR$ to the subcategory of compact objects with the duality functor, as in \Cref{ex:Betti}, further extended to a unique monoidal, compact-preserving, colimit-preserving functor $R\Gamma_{{\rm dR}, k}\colon \DA(k) \to \mathcal{D}(k)$. 
\end{exm}

\subsection{The six functor formalism}\label{s:6ff}
For any morphism $f\colon S\to T$ of schemes (resp. of rigid analytic spaces), we have natural adjunctions 
\[\begin{tikzcd}
\DA(T) \arrow[r, shift left=.75ex,"f^*"] & \DA(S) \arrow[l,shift left=.75ex,"f_*" ]
\end{tikzcd} \quad (\text{resp. }   \begin{tikzcd}\RigDA(T) \arrow[r, shift left=.75ex,"f^*"] & \RigDA(S) \arrow[l,shift left=.75ex,"f_*" ]
\end{tikzcd})\]
In fact, the assignment $\Sch\ni S\mapsto \DA(S)$ (resp.~$\mathrm{RigSpc}\ni S\mapsto\RigDA(S)$)  naturally extends into a functor 
\[\DA(-) \colon \Sch^{\op } \longrightarrow \Prlm    \quad (\text{resp. } \RigDA(-) \colon \mathrm{RigSpc}^{\op} \longrightarrow \Prlm)   \]
More is true: the presheaves $\DA(-)$ and $\RigDA(-)$ underlie a six-functor formalism, in the following sense.  For any morphism $f\colon S\to T$ of schemes (resp.~of rigid analytic spaces) locally of finite type   there are adjunctions
\[\begin{tikzcd}
\DA(S) \arrow[r, shift left=.75ex,"f_!"] & \DA(T) \arrow[l,shift left=.75ex,"f^!" ]
\end{tikzcd} \quad (\text{resp. }   \begin{tikzcd}\RigDA(S) \arrow[r, shift left=.75ex,"f_!"] & \RigDA(T) \arrow[l,shift left=.75ex,"f^!" ]
\end{tikzcd})\]
satisfying the following list of properties (see \cite{ayoub-th1,ayoub-th2,cd,agv}:
\begin{itemize}
    \item[(i)] (Functoriality). If $f$ and $g$ are composable morphisms, then $(g\circ f)^* \simeq f^*\circ g^*$ and $(g\circ f)_* \simeq g_*\circ f_*$. If both $f$ and $g$ are locally of finite type, then $(g\circ f)_! \simeq g_!\circ f_!$ and  $(g\circ f)^! \simeq f^!\circ g^!$. 
    \item[(ii)] (Proper maps). There is a natural transformation $f_!\to f_*$. If $f$ is proper, then $f_*\simeq f_!$.
    \item[(iii)] (Smooth morphisms, Relative purity). If $f$ is smooth, the functor $f^*$ admits a left adjoint $f_\sharp$. 
If $f$ is smooth of relative dimension $d$, then $f^! \simeq f^*(-)(d)[2d]$ and $f_! \simeq f_\sharp(-)(-d)[-2d]$. In particular, if $f$ is étale, then $f^*\simeq f^!$. 
\item[(iv)] (Absolute purity). If $\iota\colon S\to T$ is a regular closed immersion of  regular schemes of pure codimension $d$, then $\iota^! \one_T \simeq \one_S (-d)[-2d]$.
    \item[(v)] \label{item:loc}(Localization). Let $\iota\colon Z \to X$ be a closed immersion of schemes (resp.~a closed immersion of rigid analytic spaces \cite[Definition 1.1.14]{agv}) and let $j\colon U=X-Z\to X$ be the open complement. Then there are fiber sequences of endofunctors
    \[ j_\sharp j^* \to \id_X \to \iota_*\iota^*, \quad \iota_* \iota^! \to \id_X \to j_*j^*\]
    of $\DA(X)$ (resp.~of $\RigDA(X)$). 
    \item[(vi)](Base-change theorems). Let 
    \[\begin{tikzcd}
        Y' \arrow[r, "g'"] \arrow[d, "f'"] & Y\arrow[d, "f"] \\
        X'\arrow[r, "g"] & X
    \end{tikzcd}\]
    be a cartesian morphism of schemes (resp.~of rigid analytic spaces). 
    \begin{enumerate}
        \item If $f$ is smooth, the natural transformations
        \[ (f')_\sharp \circ (g')^* \longrightarrow g^* \circ f_\sharp, \quad f^* \circ g_* \longrightarrow (g')_*\circ (f')^* \]
        are equivalences.
        \item If $f$ is proper  then the natural transformation
        \begin{equation}\label{eq:prop_base_change} g^* \circ f_* \longrightarrow (f')_* \circ (g')^*\end{equation}
is an equivalence. If $g$ is moreover smooth, then also the natural transformation
\[ g_\sharp\circ (f')_* \longrightarrow f_* \circ (g')_\sharp\]
is an equivalence. The equivalence \eqref{eq:prop_base_change} holds also when $f$ is a quasi-compact,  and quasi-separated morphism of rigid analyitic spaces. 
\item If $f$ is locally of finite type, then the natural transformation
\[ g^* \circ f_! \longrightarrow  (f')_! \circ (g')^* \]
is an equivalence
    \end{enumerate}
    \item[(vii)] (Projection formulas). If $f\colon Y\to X$ is locally of finite type, there is an equivalence
    \[ M\otimes f_! N \simeq f_!(f^* M\otimes N) \]
    for every $M\in \DA(X)$ (resp.~in $\RigDA(X)$) and $N\in \DA(Y)$ (resp.~in $\RigDA(Y)$). In particular, when $f$ is proper, there is an equivalence $f_*(f^*M\otimes N) \simeq M \otimes f_*N$.
    \item[(viii)](Duality).  If $f\colon Y\to X$ is a smooth and proper morphism of schemes or of rigid analytic spaces. Then $f_\sharp \one_Y$ is strongly dualizable in $\DA(X)$ (resp.~in $\RigDA(X)$), with dual given by $f_*\one_Y$. More generally, if $M\in\DA(Y)$ (resp. in $\RigDA(Y)$) is dualizable, then $f_\sharp M$ also is, with dual $f_* M^\vee$.
\end{itemize}
    If $p\colon X\to S$ is a morphism locally of finite type of schemes (resp.~of rigid analytic spaces), we call the object $p_*p^*\one_S=p_*\one_X$ the \emph{cohomological motive of $X$} (relative to $S$), denoted $\mathsf{M}_S^{\rm coh}(X)$, and we call the object $p_!p^! \one_S$ the \emph{homological motive of $X$} (relative to $S$), denoted $\mathsf{M}_S(X)$. For a smooth morphism $X\to S$, this agrees with the image of $X$ under the Yoneda functor, as mentioned above. When $S$ is the spectrum (resp.~the adic spectrum) of field we will omit the subscript $S$ and simply write the homological [resp. cohomological] motive of $X$ as $\mathsf{M}(X)$ [resp.  $\mathsf{M}^{\rm coh}(X)$]. 

Similarly, we sometimes  denote by $\mathsf{M}_S(X)^{\rm c} = p_*p^!\one_S$  the \emph{motive with compact support of $X$}. It agrees with the homological motive of $X$ when $p$ is proper. 
\begin{rmk}[Localization sequences for motives over a field] Let $k$ be a field (resp.~let $K$ be a complete non-archimedean local field). Let $X\in \Sch/k$ (resp.~$X\in \mathrm{RigSpc}^{\rm qcqs}/K$) and let $\iota\colon Z\to X$ be a closed immersion of schemes (resp.~a closed immersion of rigid analytic spaces). The localization sequences of \ref{item:loc}(v) takes the  (possibly more familiar) form of a fiber sequence in $\DA(k)$ (resp.~in $\RigDA(K)$)
\[
\begin{tikzcd}
\mathsf{M}(Z)^{\rm c}\arrow[r, "\iota_*"] & \mathsf{M}(X)^{\rm c} \arrow[r, "j^*"] & \mathsf{M}(U)^{\rm c}
\end{tikzcd}
\]
If $X$ and $Z$ are both smooth, then combining  purity and duality we have the fiber sequence
\[
\begin{tikzcd}
\mathsf{M}(U)  \arrow[r, "j_*"]  & \mathsf{M}(X) \arrow[r, "\iota^!"] & \mathsf{M}(Z)(c)[2c]      
\end{tikzcd}
\]
if $c$ is the codimension of $Z$ in $X$. 
\end{rmk}

\subsection{Frobenius data}\label{sec:Frob}
We start by recording the following fact about algebraic (resp.~analytic) motives. Recall that a morphism $e\colon S'\to S$ of schemes is a universal homeomorphism if for any $Y\to S$, the base change $Y\times_S S'\to Y$ is a homeomorphism. It is well-known that this is equivalent to require that  $e$ is integral, radicial and surjective. For example, the canonical closed immersion $X_{\rm red} \to X$ is a universal homeomorphism, as it is any closed embedding $Z\to X$ defined by a nilpotent sheaf of ideals. In positive characteristic, both the absolute and the relative Frobenius are examples of universal homeomorphisms.

This notion has an exact analogue in the analytic world: we say that  a morphism of rigid analytic spaces $S'\to S$ is a universal homomorphism if it is quasi-compact, quasi-separated, surjective and radicial. See \cite[2.9.1]{agv}.

Universal homeomorphisms of schemes induce equivalences of étale sites  \cite[Exposé VIII, Théorème 1.1]{SGAIV2}. This fact has the following motivic generalization, sometimes known as the ``semi-separatedness property'' of the stable motivic categories.   
\begin{thm}[{\cite[Theorem 1.2]{AyoubEt}}, resp. {\cite[Corollary 2.9.11]{agv}}]\label{thm:semi-separat}
    Let $e : X'\to X$ be a universal homeomorphism of schemes (resp. of rigid analytic spaces with $X$ locally of finite Krull dimension). Then, the base change functor
$$e^*\colon\DA(X)\to\DA(X')\qquad \left(\text{resp. }e^* :\RigDA(X)\to\RigDA(X')\right)$$
is an equivalence. In particular, $\DA(X)\simeq \DA(X^{\Perf})$.
\end{thm}

We obtain immediately an equivalence between the associated cohomological motives of $X,X'$ as follows.
\begin{cor}\label{cor:semisep}
    Let $X\to X'$ be a universal homeomorphism of schemes over $S$. Then the map $\mathsf{M}_S^{coh}(X')\to \mathsf{M}_S^{coh}(X)$ is invertible.
\end{cor}
\begin{rmk}\Cref{thm:semi-separat} is stated for the $\Q$-linear categories $\DA(S)$ and $\RigDA(S)$, but holds more generally for the stable motivic homotopy categories $\SH(S)$ and $\mathrm{RigSH}(S)$ for
 $S$   any scheme (or rigid analytic space locally of finite Krull dimension) of characteristic $p>0$, after inverting the morphisms $\cdot p\colon E\to E$. In particular, for any \emph{oriented} motivic spectrum $\mathbf{E}\in \CAlg(\SH(S))$   representing a generalized cohomology theory on $S$-schemes, the semi-separatedness of $\SH(S)[\frac{1}{p}]$ implies that for any universal homeomorphism $f\colon Y\to X$ of $S$-schemes, the induced maps $f^*\colon \mathbf{E}^{*,*}(X)[\frac{1}{p}]\to \mathbf{E}^{*,*}(Y)[\frac{1}{p}]$ and $f_*\colon \mathbf{E}^{\rm BM}_{*,*}(Y)[\frac{1}{p}]\to \mathbf{E}^{\rm BM}_{*,*}(X)[\frac{1}{p}]$ are isomorphisms. When $\mathbf{E}=\mathrm{KGL}$ is the spectrum representing  algebraic $K$-theory, we obtain in particular that the canonical map $K(S)[\frac{1}{p}]\to K(S^{\rm perf})[\frac{1}{p}]$ is an equivalence. See \cite{kelly_morrow} for more general results in this direction. Note that this statement would be false before inverting $p$, as one can check by considering $K_1(k)$ for an imperfect field $k$. 
\end{rmk}
We can specialize the above Theorem to construct a canonical ``Frobenius enrichment'' for motives over schemes or analytic varieties defined over a field of positive characteristic $p$. We first observe that for $S$ any scheme (resp.~a rigid analytic space locally of finite Krull dimension)  of characteristic $p$ and $X$ smooth over $S$, the relative Frobenius defines a natural transformation $X\to X\times_\varphi S = \varphi_S^* X$ of $S$-schemes (resp.~of analyitic spaces over $S$). In particular, there is a natural transformation $\varphi^*\colon \id \to \varphi_S^*$ of endofunctors of $\DA(S)$ (resp.~of $\RigDA(S)$), which is in fact an equivalence in light of \Cref{thm:semi-separat}. If we denote by $\DA(S)^{\varphi^*}$ (resp.~by $\RigDA(S)^{\varphi^*}$) the category of \emph{homotopy fixed points of $\varphi^*$}, i.e., the equalizer of $\varphi^*$ and the identity functor $\id$, computed in $\Prloo$, we obtain the following Corollary.

\begin{cor}[{\cite[Corollary 2.26]{lbv}}]
    Let $S$ be a scheme (resp. a rigid analytic space locally of finite Krull dimension) of characteristic $p$ and $\varphi$ be the  absolute Frobenius map $\varphi\colon S\to S$. There is   a canonical functor
    $$
    \DA(S)\to \DA(S)^{\varphi^*} \qquad \left(\text{resp. }\RigDA(S)\to\RigDA(S)^{\varphi^*}\right)
    $$
    which informally sends $\mathsf{M}_S(X)$ to $(\mathsf{M}_S(X),\varphi_S\colon \mathsf{M}_S(X)\stackrel{\sim}{\to} \mathsf{M}_S(X\times_{\varphi}S))$ where $\varphi_S$ is the relative Frobenius map.
\end{cor}

\subsection{The spreading out}
The functor $S\mapsto \RigDA(S)$ has the property of converting  limits of certain
cofiltered inverse systems of rigid analytic spaces into filtered colimits of presentable $\infty$-categories. 
\begin{thm}[{\cite[Theorem 2.8.15]{agv}}]\label{spreading}
Let $(X_{i})_{i\in I}$ be a cofiltered inverse system of 
rigid analytic spaces, with quasi-compact and quasi-separated 
transition maps, and admitting a weak limit $X_\infty$ in the sense of Huber (see \cite[Definition 2.4.2]{huber}). We further assume that $X_\infty$ and the $X_i$'s are locally of finite Krull dimension. 

Then, the obvious functor 
\begin{equation}
\label{eq-thm:anstC-v2-1}
\underset{i}{\varinjlim}\,\RigDA^{}_{}
(X_{i}) \to \RigDA^{}_{}(X_\infty),
\end{equation}
where the colimit is taken in $\Prl$, is an equivalence.
\end{thm}

\begin{rmk}\label{rmk:spreading_ex}
    We like to call this a ``spreading out'' theorem as it ultimately says that any compact motive over $X_\infty$ has a model over some $X_i$. For example, by putting $X_i$ to be $\Spa L$ with $L$ ranging over finite extensions of $\Q_p$ in $\C_p$, we deduce 
    $$
    \RigDA(\C_p)\simeq \varinjlim_{[L:K]<\infty}\RigDA(L).
    $$
    We remark that the analogous statement would not hold if we simply considered the big \'etale sites over $\C_p$ and over the $L$'s (for example, a map $\Spa\C_p\to\B^1_{\C_p}$ hitting an element of $\C_p$ which is trascendental over $\Q_p$ doesn't admit any model over any $L$).
    
    Also, by putting  $X_i$ to be the affinoid open neighborhoods of a $K$-rational point $\Spa K=x\in S(K)$ in some rigid variety  $S$ over $K$, we deduce
    $$
    \RigDA(x)\simeq \varinjlim_{x\in U\subset S}\RigDA(U).
    $$
    In this case, the explicit quasi-inverse of the canonical pull-back functor $x^*$ is given by $\Pi^*$ i.e. the pull-back along the structural morphism $\Pi\colon U\to \Spa K$.  Informally, this says that any compact motive $M$ over $S$ becomes constant around each rational point, i.e. for any sufficiently small open neighborhood $j\colon U{\hookrightarrow}S$ of $x$, one has $j^*M\simeq \Pi^*x^*M$.  
    
    We note that the ``spreading out'' formula in this last case (a rational point $x$ as the limit of its open neighborhoods)  follows immediately from the localization triangle (Section \ref{item:loc}(v)) and plays an important role in proving finiteness results of de Rham cohomologies (see \cite[Theorem 4.46]{lbv}, and even \cite{Vezz-dR} for an algebraic analogue).
\end{rmk}

The following is basically a restatement of the spreading out theorem, in terms on cohomological motives.
\begin{prop}\label{prop:spreadingonmotives}
 Let $(\pi_i\colon X_i\to S)_{i\in I}$ be a cofiltered inverse system of rigid analytic spaces locally of finite Krull dimension which are qcqs over $S$ 
  admitting a weak limit $\pi_\infty\colon X_\infty\to S$ over $S$ which is locally of finite Krull dimension. Let $f_{ij}\colon X_i\to X_j$ be the natural transition maps, with $i,j\in I\cup\{\infty\}$. Then for any $M\in\RigDA(X_0)$,  the canonical map in $\RigDA(S)$
 $$
 \varinjlim \pi_{i_*}f_{i0}^*M \to \pi_{\infty*}f_{\infty0}^*M
 $$
 is invertible in $\RigDA(S)$.
 \end{prop}
 \begin{proof}
     We let $C$ be a compact object in $\RigDA(S)$. By \cite[Proposition 2.4.22]{agv}, the functors $f_{ij}^*$ and $\pi_i^*$ are compact-preserving.  By \Cref{spreading}, we know that $\RigDA(X_\infty)$ is equivalent to $\varinjlim\RigDA(X_i)$ in $\Prloo$. As such, for any two objects $A,B$ in $\RigDA(X_0)$ with $A$ compact, the mapping complex $\map(f_{\infty0}^*A,f_{\infty0}^*B)$ can be computed as $\varinjlim \map(f_{i0}^*A,f_{i0}^*B)$. We then deduce the following sequence of canonical equivalences:
     \begin{align*}
     \map(C,\varinjlim \pi_{i_*}f_{i0}^*M )&\simeq \varinjlim\map (f_{i0}^*\pi_{0}^*C,f_{i0}^*M)\\&\simeq \map(\pi_{\infty}^*C,f_{\infty0}^*M)\\&\simeq \map(C,\pi_{\infty*}f_{\infty0}^*M).
     \end{align*}
     As $\RigDA(S)$ is compactly generated (see \cite[Proposition 2.4.22]{agv}), this is enough to prove our claim.
 \end{proof}

 By putting $M=\one$ and $S=\Spa K$ we deduce:

 \begin{cor}\label{cor:Xinfty=colim}
     Let $X_\infty\sim \varprojlim X_i$ be a qcqs weak limit of qcqs adic spaces over $K$,  with finite Krull dimension. 
     The canonical map $\varinjlim \mathsf{M}^{\rm coh}(X_i)\to \mathsf{M}^{\rm coh}(X_\infty)$ is an equivalence in $\RigDA(K)$.  
         %
 \end{cor}
The previous result is particularly useful when one knows that the maps in the direct system induce equivalences in the homotopy category, e.g. if they are Frobenius lifts (see \Cref{sec:Frob}). If they are finite, we also deduce a special relation between the motive of the limit and each motive of the elements in the tower.

\begin{cor}\label{dirfact}
   Let $(X_i,f_{ij})$ be a projective system of qcqs smooth rigid varieties over $K$, whose transition maps $f_{ij}$ are finite surjective, and let $X_\infty$ be a qcqs adic space such that $X_\infty\sim \varprojlim X_i$.  The motive $\mathsf{M}^{\rm coh}(X_0)$ is a direct summand of the motive $\mathsf{M}^{\rm coh}(X_\infty)$. 
\end{cor}

 \begin{proof}
     We may and do assume that all the spaces involved are connected. We recall that, based on the equivalence between motives with and without transfers $\RigDA(K)\simeq\RigDM(K)$, the finite transition maps induce motivic maps $f_{ij}\colon \mathsf{M}( X_i)\to \mathsf{M}(X_j)$ as well as transpose correspondences $f_{ij}^T\colon \mathsf{M}(X_j)\to \mathsf{M}(X_i)$ such that $f_{ij}\circ f_{ij}^T=\deg(f_{ij})\cdot\id$. 
     We then deduce that the motive $\mathsf{M}(X_0)$ is a direct summand of each $\mathsf{M}(X_i)$, and dually that the motive $\mathsf{M}^{\rm coh}(X_0)$ is a direct summand of each $\mathsf{M}^{\rm coh}(X_i)$. We obtain the claim passing to the colimit $\varinjlim \mathsf{M}^{\rm coh}(X_i)\simeq \mathsf{M}^{\rm coh}(X_\infty)$. 
     \end{proof}

As an extra corollary of the spreading out result, we get the existence of ``motivic tubular neighborhoods'' for hypersurfaces defined as the zero-locus of a smooth map (see \cite[Corollary 4.5]{pWM}).

\begin{cor}\label{cor:tubes_spread}
    Let $f\colon X\to S$ be a smooth map of smooth rigid varieties over $K$ and $s\in S(K)$ a  rational point. There exists a final family of open neighborhoods $U$ of $s$ for which $\mathsf{M}^{\rm coh}(X_U)\simeq \mathsf{M}^{\rm coh}(X_s)$ in $\RigDA(K)$.
\end{cor}
\begin{proof}
    The point $s$ has a final family of open neighborhoods $U$ which are all isomorphic to poly-discs. In light of  \Cref{rmk:spreading_ex}, we may and do suppose that $\mathsf{M}_U(X_U)\simeq \mathsf{M}_U(X_s\times U)$. This implies that $\mathsf{M}_K(X_U)\simeq \mathsf{M}_K(X_s\times U)=\mathsf{M}_K(X_s)$ and hence that $\mathsf{M}^{\rm coh}(X_U)=\mathsf{M}^{\rm coh}(X_s)$.
\end{proof}

\section{From characteristic $p$ to mixed characteristic}
\label{sec:fromktoK}
The aim of this section is to prove \Cref{intro:rig} that is, to use homotopy methods to give a new alternative definition of relative rigid cohomology, and even deduce new results about it. For the absolute case, see \cite{vezz-bang, vezz-MW}.

\subsection{Formal motives vs algebraic motives}
We first show how Elkik and Arabia's solutions to the (local) ``lifting problems'' in the Monsky--Washnitzer approach to rigid cohomology can be condensed into suitable motivic equivalences.

The first result is due to Ayoub (see e.g., \cite[Corollary 1.4.29]{ayoub-rig}), and is a consequence of the localization theorem.
\begin{thm}[{\cite[Theorem 3.1.10]{agv}}]
\label{thm:formal-mot-alg-mot}
Let $\mfS$ be a formal scheme. The functors
\begin{equation}
\sigma^*:\FDA^{}(\mathfrak{S})
\rightleftarrows \DA_{}^{}
(\mathfrak{S}_{\sigma}):\sigma_*.
\end{equation}
induced by taking special fibers, are equivalences of $\infty$-categories.
\end{thm}
Indeed, for a formal scheme $\mathfrak{S}$, the complement of $\mathfrak{S}_\sigma$ in $\mathfrak{S}$ is empty and so by the localization triangle we obtain that the natural transformation $\id \to \sigma_* \sigma^*$ is invertible.

There is also a log-generalization of the equivalence provided by \cref{thm:formal-mot-alg-mot} above, based on a modification of the argument to the logarithmic setting. For $S$ a log scheme (resp. $\mathfrak{S}$ a log formal scheme), we let $\mathrm{logDA}^{\rm eff}({S})$ (resp. $\mathrm{logFDA}^{\rm eff}(\mathfrak{S})$) be the full monoidal $\infty$-subcategory of $\A^1$-invariant objects in the category $\mathrm{Shv}_{s\acute{e}t}(\Sm/{S}; \mcD(\Q))$ (resp. $\mathrm{Shv}_{s\acute{e}t}(\Sm/\mathfrak{S}; \mcD(\Q))$) of strict-étale hypersheaves on the category of log smooth log schemes over $S$ (resp. log formal schemes over $\mathfrak{S}$). This is the exact generalization to the log setting of \Cref{def:mot_sheaves}. Formally inverting the Tate twist $\one(1)$ we obtain the non-effective version $\mathrm{logFDA}^{}({S})$ (resp. $\mathrm{logFDA}^{}(\mathfrak{S})$). See \cite[Section 2.2]{pWM} for an extended discussion,  \cite{kato_log} for basic results about fine log structures and \cite[Appendix A]{koshikawa} for non-necessarily fine log structures on formal schemes. 

\begin{thm}[{\cite[Theorem 2.22]{pWM}}]\label{thm:nil-invariance_log}
    	Let $\mfS$ is a quasi-coherent integral log formal scheme. Assume that at least one of the following holds:
			\begin{enumerate}
\item The log structure $M_{\mfS}$ is fine;
\item 		The underlying scheme 	$\underline{\mfS}$ has 	finite topological Krull dimension.
			\end{enumerate}%
		 Then the functors
\begin{equation}
\sigma^*:\mathrm{logFDA}^{}(\mathfrak{S})
\rightleftarrows \mathrm{logDA}_{}^{}
(\mathfrak{S}_{\sigma}):\sigma_*.
\end{equation}
induced by taking special fibers, are equivalences of $\infty$-categories.
\end{thm}

\subsection{Dagger motives vs analytic motives}

The possibility to ``lift'' varieties over $\mfS_\sigma$ to (formal) motives over $\mfS$ and hence to (analytic) motives over $\mfS_{\Q_p}$ is not yet enough to produce a motivic $p$-adic cohomology theory ``\`a la de Rham'', which is finite dimensional on objects of geometric origin (for example on a smooth quasi-compact rigid analytic variety
$X$ over $K$). The problem being, that the de Rham complex of the closed disc $\B^1$ has non-trivial higher cohomology. 

This obstacle is classically overcome as follows. Any smooth rigid analytic variety $X$ can be embedded, locally on $X$, as a strict open subvariety of a larger smooth variety $X'$. This choice is called a ``dagger variety structure on $X$'' and denoted by $X^\dagger$. One can then replace the de Rham complex $\Omega_{X/K}$ by its subcomplex $\Omega_{X^\dagger/K}$ given by those differential forms that are defined over some strict neighborhood of $X$ inside $X'$. One can make a similar construction relatively over a (dagger) base $S$.

Just like the case of formal lifts, the choice of dagger variety structures are only local, and canonical up to homotopy. It comes as no surprise that this admits a motivic restatement.

\begin{thm}[{\cite[Theorem 4.23]{vezz-MW}}]\label{thm:dagger_or_not_dagger}
    Let $\mcS^\dagger $ be a dagger variety. The functors
\begin{equation}
l^*:\mathrm{RigDA}^{}(\mcS^\dagger)
\rightleftarrows \mathrm{RigDA}_{}^{}
(\mcS):l_*.
\end{equation}
induced by taking the underlying rigid analytic variety, are equivalences of $\infty$-categories.
\end{thm}

We can then introduce a suitable generalization of the overconvergent site above a scheme $S/k$ as follows.

\begin{dfn}
    Let $S$ be a scheme over a perfect field $k$ of characteristic $p$. The category $\Fr/S$ of generalized proper frames over $S$ is given by triples $(\mfT,f\colon\mfT_\sigma\to S, \mfT_{\Q_p}^\dagger)$ where  $\mfT$ is a formal scheme over $W(k)$ (not necessarily adic over it), $f$ is a morphism over $k$ and $\mfT_{\Q_p}^\dagger$ is a dagger structure on the $\Q_p$-generic fiber of $\mfT$.
\end{dfn}

Our notation for the name is justified by the following remark: any proper frame over $S$ in the sense of Berthelot $(T\subset \bar{T}\subset\mfP, f\colon T\to S)$ given by an open inclusion into a closed subscheme of a proper formal scheme $\mfP$ induces a generalized proper frame by taking $(\mfU, f, ]T[_{\mfP}\Subset]\bar{T}[_{\mfP})$, where $\mfU$ is the open formal subscheme of $\mfP^{\wedge\bar{T}}$ whose special fiber is $T$, and whose $\Q_p$-generic fiber is the tube $]T[_{\mfP}$.

We can then mix together \Cref{thm:formal-mot-alg-mot} and \Cref{thm:dagger_or_not_dagger} to finally create a Monsky--Washnitzer relative realization (with a Frobenius structure!).

\begin{thm}[{\cite[Proposition 3.19]{ev}}]
    Let $S$ be a scheme over $k$. There is a canonical functor
    $$
    \mathcal{MW}_S\colon\DA(S)\to \left(\lim_{(\mfT,f, \mfT_{\Q_p}^\dagger)\in\Fr/S}\RigDA^\dagger(\mfT_{\Q_p}^\dagger)\right)^\varphi.
    $$
By taking $S=\Spec k$, we may in particular obtain the ``absolute'' realization
    $$
    \mathcal{MW}_k\colon\DA(k)\to\RigDA^\dagger(W(k^{})[1/p])^\varphi.
    $$
\end{thm}

\subsection{de Rham and rigid cohomology}
As we have finally completed the journey from characteristic $p$ to characteristic $0$ (plus a dagger structure), we have all the ingredients to define a ``de Rham-like'' cohomology theory.

\begin{prop}[{\cite[Corollary 4.39]{lbv},\cite[Corollary 3.20]{ev}}]
    The functor $$\mcX^\dagger\mapsto\underline{\Omega}_{\mcX^\dagger/\mcS^\dagger}$$
    induces a motivic realization functor
    $$
    \RigDA(\mcS^\dagger)\to \mcD(\mcS^\dagger)^{op}
    $$
    where on the right hand side we consider the category of overconvergent solid modules on $\mcS^\dagger$. In particular, we deduce (absolute and relative) rigid cohomology functors
    $$
   R\Gamma_{\rig} \colon\DA(k)^{op}\to \mcD(W(k)[1/p])^\varphi\qquad R\Gamma_{\rig}(-/S)\colon \DA(S)^{op}\to \left(\lim \mcD(\mfT^\dagger_{\Q_p})\right)^\varphi.
    $$
\end{prop}

We may then prove specific properties of the realization of smooth and proper varieties by using the fact that their associated motives are dualizable: one uses crucially the explicit description of dualizable solid modules as (classical) perfect complexes, see \cite{andreychev}.

\begin{cor}[{\cite[Theorem 4.46]{lbv},\cite[Corollary 3.20]{ev}}]
    If $M\in\DA(S)$ is dualizable, then $R\Gamma_{\rig}(M/S)$ is a pefect complex on each frame. Moroever, its cohomology groups $R^i\Gamma_{\rig}(M/S)$ have a canonical structure of overconvergent $F$-isocrystals. 
\end{cor}

\begin{rmk}
    In case $M=\mathsf{M}_S(X)$ for some smooth and proper variety $X/S$,  the  overconvergent sheaves underlying the $R^i\Gamma(M/S)$'s can be compared to Berthelot's rigid cohomology groups. The result above answers positively to a conjecture by Berthelot, which was previously known only in the liftable case, or considering variants of it (see in particular \cite{caro3, shiho_relative3}).
\end{rmk}

\section{Motivic nearby cycles and applications}\label{sec:fromKtok}

\subsection{Analytic motives vs  algebraic motives with monodromy}
In the classical literature (see \cite{illusie}) monodromy operators are certain nilpotent endomorphisms between cohomology groups $H^*(X)$ (e.g. de Rham cohomology, $\ell$-adic cohomology), where $X$ arises as a fiber of a family over a disk, or $X$ is a variety defined over a local field. They come from the action of a  generator of a fundamental group $\pi^1(\Delta^*)\simeq\Z$ or a topological generator of tame inertia groups (isomorphic to $\widehat{\Z}$). If one equips the groups $H^*(X)$ with suitable extra structures (in the de Rham case: a mixed Hodge module structure, in the $\ell$-adic case: the structure of a representation over $\Gal(k)$) then it makes sense to define (positive and negative) twists of them, and it turns out that the monodromy maps are actually morphisms of the form
$$
N\colon H^*(X)\to H^*(X)(-1)
$$
hence their nilpotence (for weight reasons).

We can then define a category of monodromy operators as follows:

\begin{dfn}\label{def:def_DN}
    Let $\catC$ be   stable \icat, and $(-1) : \catC\to\catC$ be an exact functor. We  let $(-n)$ be the iterated functor $(-1)^n$.
    \begin{enumerate}
        \item We let $\catC^{(-1)}_{lax}$ to be the \icat whose objects are pairs $(X,f)$ with $X\in\catC$ and $f\colon X\to X(-1)$. Mapping complexes are defined in the obvsious way, i.e. $$
        \map((X,f),(Y,g)))=\fib(\map(X,Y)\xto{f^*(-1)-g_*}\map(X,Y(-1)).
        $$
        \item We define $\catC_N$ as the
full sub-\icat of $\catC^{(-1)}_{lax}$ spanned by objects $(X, f )$ with ind-nilpotent morphism $f$, that is, such that $$\colim (X\to X(-1)\to X(-2)\to\cdots)=0.$$
    \end{enumerate}
\end{dfn}

The following facts are proved in \cite[Section 2]{bgv}.
\begin{rmk}[{See \cite[Remark 2.20]{bgv}}]
    If $\catC$ is compactly generated, then also   $\catC^{(-1)}_{lax}$ is. Its compact objects are given by pairs $(X,f)$ with $X$ compact, and $f$ nilpotent.
\end{rmk}
\begin{rmk}[{See \cite[Remark 2.23]{bgv}}]
    The category $\catC^{(-1)}_{lax}$ underlies a structure of a symmetric monoidal category, as soon as $\catC$ does and the functor $(-1)$ is lax-monoidal. Informally, one has $(X,f)\otimes(Y,g)=(X\otimes Y, f+g)$.
\end{rmk}
\begin{rmk}\label{rmk:N=0}
    We note that there is a natural adjunction
$$
(N=0)\colon\catC\rightleftarrows \catC_N\colon \fib(N)
$$
which, in good situations, satisfies the hypotheses of a Barr--Beck--Lurie functor, thus identifiying the category $\catC_N$ with a category of modules $\Mod_{\fib(N)(\one)}\catC=\Mod_{\one\oplus(-1)[-1]}(\catC)$. 
\end{rmk}
Let $K$ be a complete non-archimedean field with a perfect residue field $k$. Consider the generic fiber functor
\[\xi\colon \DA(k) \simeq \FDA(\mathcal{O}_K) \xrightarrow{(-)^{\rm rig}} \RigDA(K)\]
where the first equivalence is provided by \Cref{thm:formal-mot-alg-mot} and $(-)^{\rm rig}$ is the functor induced by
sending a formal scheme to its associated rigid space (see \cite[Notation 3.8.14]{agv}). The main theorem of \cite{bgv} states the following. 
\begin{thm}[{\cite[Corollary 4.14]{bgv}}]
\label{thm:DA_N_is_RigDA}\label{thm:bgv}
    Let $K$ be as above and assume that it has a value group of (rational) rank 1. Let  $\RigDA(K)^{\gr}$ be the essential image of the functor
    $\xi\colon \DA(k)\to\RigDA(K).$
    Depending on a choice of a pseudo-uniformizer $\varpi\in K$, the functor above exhibits an equivalence between $\DA(k)_N$ and $\RigDA(K)^{\gr}$. 
\end{thm}

\begin{dfn}
    Notation-wise, we shall write the image in $\DA(k)_N$ of an object $M\in\RigDA(K)^{\gr}$ via the previous equivalence as $(\Psi M, N_M)$.
\end{dfn}

\begin{proof}
We can highlight the main points of the proof as follows.

{\it Step 1:} We note that the functor $\xi\colon  \DA(k)\to \RigDA(K)^{\gr}$ satisfies the hypotheses of the Barr--Beck--Lurie theorem. Indeed, it is a left adjoint which sends a set of compact generators to a set of compact generators (more or less by definition), which implies that it preserves colimits and its right adjoint preserves filtered colimits. The  projection formula $\chi(\xi M\otimes N)\simeq M\otimes \chi N$ follows formally from the fact that all compact objects in $\DA(k)$ are dualizable.

{\it Step 2:} Because of Step 1, we deduce that $\xi$ exhibits a a monoidal equivalence $\RigDA(K)^{\gr}\simeq \Mod_{\chi\one}(\DA(k))$. Thanks to the computations of \cite{agv}, we have a concrete description of the algebra $\chi\one$ as the split square-zero extension $\one\oplus(-1)[-1]$. (This identification depends on a choice of a pseudo-uniformizer.) In particular, we can alternatively identify $\Mod_{\chi\one}(\DA(k))$  with $\catC_N$, as shown above. 
\end{proof}

\begin{rmk}\label{rmk:identifadj}
    We note that under the equivalence $\DA(k)_N\simeq \RigDA(K)^{\gr}$ of \Cref{thm:DA_N_is_RigDA}, the adjunction
    $$
\xi\colon    \DA(k)\rightleftarrows\RigDA(K)^{\gr}\colon \chi 
    $$
    is identified with the adjunction of \Cref{rmk:N=0}:
        $$
(N=0)\colon    \DA(k)\rightleftarrows\DA_N(K)\colon \fib (N). 
    $$
\end{rmk}
\begin{rmk}
    Step 1 of the proof above can be generalized over a base, giving rise to an equivalence between the essential image of $\xi\colon\FDA(\mfS)\to\RigDA(\mfS_\eta)$ and $\Mod_{\chi\one}(\DA(\mfS_\sigma)$. In this setting, the projection formula is not automatic, and is the core result of \cite[Section 3.6]{agv}.
\end{rmk}

The previous theorem is only mildly useful if one doesn't have a good handle of the categories $\RigDA(K)^{\gr}$. Fortunately, de Jong's alteration of singularities and some explicit homotopies allow one to conclude the following.
\begin{prop}{\cite[Remarks 4.6, 4.7]{bgv}}\label{prop:semiss_is_gr}
    Let $K$ be a non-archimedean field.
    \begin{enumerate}
    \item If $X/K$ is a quasi-compact smooth rigid variety having a semi-stable formal model (more generally, a pluri-nodal model in the sense of \cite[Definition 1.1]{berk-contr}) then $\mathsf{M}(X)\in\RigDA(K)^{\gr}$.
        \item If $K$ is algebraically closed, then $\RigDA(K)=\RigDA(K)^{\gr}.$ In particular, for any $M\in\RigDA(K)$ which is compact, there exists a finite extension $K'/K$ such that $M_{K'}\in\RigDA(K')^{\rm gr}$.
    \end{enumerate}
\end{prop}
\begin{rmk}
    Passing to finite extensions of the base field in order to have an alteration from a variety admitting a  semi-stable formal model is a standard technique which is used to bootstrap results from the semi-stable case to the general case. Similarly, we will focus on the categories $\RigDA(K)^{\gr}$, as any compact motive $M\in \RigDA(K)$ lies in them, up to some finite field extension (use the previous proposition and the spreading-out).
\end{rmk}

\subsection{Revisiting the tilting equivalence}
\Cref{{thm:DA_N_is_RigDA}} has the following straightforward corollary.
\begin{cor}\label{cor:tilting}
    Let $K$ be a perfectoid field with rational rank $\Q$. Then there is an equivalence
    $$
    \RigDA(K)\simeq\RigDA(K^\flat)
    $$
\end{cor}
\begin{proof}
    If $K$ is algebraically closed, then both categories are equivalent to $\DA_N(k)$ where $k$ is their common residue field. For the general case, observe that this equivalence is equivariant with the action of use $\Gal(K)\simeq\Gal(K^\flat)$.
\end{proof}

The result above is a special case of an equivalence which holds in much greater generality, and whose proof is based on the spreading out theorem, and some explicit computations of mapping spaces:
\begin{thm}\label{thm:tilting}\label{tilt}
    Let $P$ be a perfectoid space over $K$. There is an equivalence $\RigDA(P)\simeq\RigDA(P^\flat)$ which is compatible with pull-back functors. In particular, under the equivalence $\RigDA(K)\simeq \RigDA(K^\flat)$, the cohomological motive $\mathsf{M}^{\rm coh}(P)$ of a perfectoid space $P/K$ corresponds  to $\mathsf{M}^{\rm coh}(P^\flat)$.
\end{thm}

\begin{proof}The first part is \cite[Theorem 5.13]{lbv}. For the last point, it suffices to point out that the identification $\RigDA(P)\simeq\RigDA(P^\flat)$ is also compatible with the   functors $f_*$, as it is compatible with their left adjoint functors $f^*$.
\end{proof}
We note that there is no analogue of the previous result for \emph{homological} motives: if $f\colon P\to K$ is perfectoid, there is no functor $f_!$ as $f$ is not locally of topologically finite type, in general.
\begin{rmk}
The equivalence of \Cref{cor:tilting} is    a special case of the one of \Cref{thm:tilting}. In order to see this, it suffices to show that the tilting equivalence is compatible with the functors $\DA(k)\to\RigDA(K)$ and $\DA(k)\to\RigDA(K^\flat)$, as shown in \cite{vezz-tilt4rig}.
\end{rmk}
\begin{rmk}
In this remark, we review one of the main results of \cite{pWM} in which the tilting equivalence is shown to be compatible with the theory of log-motives on the log point. 

Let $C$ be the completion of an algebraic
closure of a discrete valuation field,  let $\mathcal{O}_C$ be its valuation ring and $k$ its residue field. We can canonically equip $\Spf(\mathcal{O}_C)$ with a logarithmic structure, denoted $\mathcal{O}_C^\times$, by considering the inclusion of the submonoid of non-zero elements $\mathcal{O}_C - \{0\} \subset \mathcal{O}_C$. The residue field $k$ also acquires a logarithmic structure as the so-called log point $k^0 = (k, k^* \oplus_{\mathcal{O}_C^*} \mathcal{O}_C-\{0\})$. Note that this is simply the induced log structure for the embedding $\Spec(k) \to \Spf(\mathcal{O}_C)$: in terms of maps of monoids, it is the log structure associated to the pre-log structure $\mathbb{Q}_{\geq0}\to k, 0\neq m\mapsto 0$, where $\mathbb{Q}_{\geq0}$ is the valuation group of $C$. 

We can consider the category $\Sm^{\rm ss}_{k^0}$ of \emph{semistable}   log schemes $X=(\underline{X}, \mathsf{M}_X)$ over $k^0$: it is the full subcategory of all log smooth log schemes over $k^0$ consisting of those $X$ such that, Zariski locally on $\underline{X}$, there exist a strict étale map $X \to \Spec(k[X_0,\ldots, X_n]/(X_1\cdots X_r))$, equipped with the standard log structure induced by $\mathbb{N}^r \to k[X_0,\ldots, X_n]/(X_1\cdots X_r)$, $e_i \mapsto X_i$, for some $1\leq r\leq n$. See \cite[Definition 2.11(1), (2)]{pWM} for more details. 

Write $\mathrm{log}\DA^{\rm ss}(k^0)$ for the category of semistable log motives over $k^0$: this is the full subcategory of $\log\DA(k^0)$ spanned by motives of semistable log varieties over $k^0$. The equivalence of \Cref{thm:nil-invariance_log} restricts in particular to an equivalence 
\[ \mathrm{log}\DA^{\rm ss}(k^0) \simeq \mathrm{log}\FDA^{\rm ss}(\mathcal{O}_C^\times),\]
where the right-hand side is defined in the obvious way. By construction, the log structures on semistable formal schemes over $\mathcal{O}_C$ are in particular \emph{vertical}, i.e., the sheaf of monoids $\mathsf{M}_\mathfrak{X}$ on a semistable formal scheme $\mathfrak{X} \to \Spf(\mathcal{O}_C)$ becomes a sheaf of groups after inverting a pseudo-uniformizer $\varpi$. In particular there is a well-defined functor 
 \[ \Sm^{\rm ss}_{\mathcal{O}_C^\times} \to  \RigSm_C, \quad \mathfrak{X} = (\underline{\mathfrak{X}}, \mathsf{M}_{\mathfrak{X}})\mapsto \underline{\mathfrak{X}}^{\rig} \]
where $\underline{\mathfrak{X}}^{\rig}$ is the rigid generic fiber. This passes to the category of motives, providing a functor 
\[ \mathcal{MW}_k^{\rm log} \colon \mathrm{log}\DA^{\rm ss}(k^0) \simeq \mathrm{log}\FDA^{\rm ss}(\mathcal{O}_C^\times) \longrightarrow \RigDA(C)\]
that we can call the  log Monsky–Washnitzer functor.  This is compatible with the tilting equivalence, in the sense that if  $C$ is  perfectoid field with tilt $C^\flat$, then there is a commutative diagram
\[
\begin{tikzcd}
& \mathrm{log}\FDA^{\rm ss}(\mathcal{O}_C^\times) \arrow[r] & \RigDA(C)\arrow[dd, dash, "\simeq"]\\ 
  \mathrm{log}\DA^{\rm ss}(k^0) \arrow[rd, dash, "\simeq"] \arrow[ru, dash, "\simeq"'] &   &\\
  & \mathrm{log}\FDA^{\rm ss}(\mathcal{O}_{C^\flat}^\times)\arrow[r] & \RigDA(C^\flat)
\end{tikzcd}
\]
see \cite[Proposition 2.31]{pWM}. In fact, the horizontal arrows in the previous diagram are actually localizations with respect to the rig-étale topology. 
\end{rmk}

\begin{rmk}This can be made even more explicit:  one can consider the category of log motives $\mathrm{log}\DA^{\rm ss}_{\rm div}(k^0)$ for the \emph{dividing} étale topology \cite[Section 3.3]{mot_log}, i.e., the topology generated by strict étale maps together with universally surjective proper log étale monomorphisms (informally, locally given by blow-ups associated to subdivisions of fans). Then by \cite[Theorem 3.3]{ParkPsi} and \Cref{thm:DA_N_is_RigDA}, there are equivalences \[\mathrm{log}\DA^{\rm ss}_{\rm div}(k^0) \simeq \DA(k)_N \simeq \RigDA(C)\]
This allows one to  consider log motives as algebraic avatars of analytic motives. In equi-characteristic zero, Park further shows that the motivic nearby cycle functor $\Psi$ can be computed via the functor $\Psi^{\rm log}$. See \cite[Definition 2.2]{ParkPsi}.
\end{rmk}
\subsection{A quick detection principle}
As another corollary of \Cref{thm:DA_N_is_RigDA}, we point out that one can immediately extend (monoidal) functors $F\colon \DA(k)\to\catD^\otimes$ to functors from $\RigDA(K)^{\gr}$. More precisely
let $F\colon\DA(k)\to \catD^\otimes$ be a  colimit-preserving monoidal functor between monoidal stable presentable  \icats. Let $u=F(\one(-1))\in\catD$, and write $\catD_N$ for the category of \Cref{def:def_DN}-(2) with respect to the exact endofunctor $(-1):=-\otimes u$.  The choice of a pseudo-uniformizer~$\varpi\in\mcO_K$ gives rise to a monoidal extension
				\begin{equation}\label{eq:def_Fhat}
				\widehat{F}\colon\RigDA(K)^{\rm gr}\to \catD_N
				\end{equation}
				endowed with equivalences $\widehat{F}\circ\xi\simeq (N=0)\circ F$ and $\pi\circ\widehat{F}\simeq F\circ \Psi$. Informally, it sends the motive $M$ to $(F(\Psi M), F(N_{\Psi M}))$.
                
                We can be even more precise than that, by classifying these possible extensions explicitly.

\begin{prop}
    Let $F\colon \DA(k)\to\catD^\otimes$ be a  functor of monoidal presentable categories. There is an equivalence
    $$
    \Map_{\DA(k)^\otimes/}(\RigDA^{\gr}(K)^\otimes,\catD^\otimes)\simeq \Map_{\catD}(F(\one(-1)[-1]),\one).
    $$
    In particular, in case $\catD=\mcD(\mcA)$ for some monoidal abelian category $\mcA$, there is an equivalence between $\pi_0\Map_{\DA(k)^\otimes/}(\RigDA^{\gr}(K)^\otimes,\catD^\otimes)$ and $\Ext^1(F(\one(-1)),\one)$ under which the functor $F\circ \Psi$ corresponds to $0$.
\end{prop}
\begin{proof}For the first claim, we let $G$ be the right adjoint to $F$ and we note the following identifications (using \Cref{thm:bgv} and the Barr--Beck--Lurie theorem \cite[Proposition 5.29]{MNN}, see the proof of \cite[Proposition 4.52]{bgv})
$$
\begin{aligned}
    \Map_{\DA(k)^\otimes/}(\RigDA^{\gr}(K)^\otimes,\catD^\otimes)&\simeq  
 \Map_{\DA(k)^\otimes/}(\Mod_{\one\oplus\one(-1)[-1]}\DA(k)^\otimes,\catD^\otimes)\\&\simeq
 \Map_{\CAlg(\DA(k))}(\one\oplus\one(-1)[-1],G\one)\\&\simeq 
  \Map_{\DA(k)}(\one(-1)[-1],G\one)\\&\simeq
   \Map_{\catD}( F\one(-1)[-1],\one).
   \end{aligned}
$$
    For the last sentence, we might consider the case $\catD=\DA(k)$ and remark that the base change along the augmentation $\one\oplus\one(-1)[-1]\to\one$, induced by the zero map $\one(-1)[-1]\to\one$,  corresponds to the functor $\Psi$.
\end{proof}
\begin{exm}Write $K_0 =W(k)[1/p]$.
    If $\mcA$ is the category of $(\varphi,N)$-modules and $F(\one(-1))=\one(-1)$, then $\Ext^1(\one(-1),\one)=K_0$ having as $K_0$-basis the ``Kummer extension'', i.e. the  $\varphi$-module $\one\oplus\one(-1)$ with monodromy $\left(\begin{smallmatrix}
        0&1\\0&0
    \end{smallmatrix}\right)$.
\end{exm}

In general, the object in $\Ext^1(F(\one(-1)),\one)$ associated to $F$ can be read out of the image of the \emph{Kummer motive} $\mcK\in\RigDA^{\gr}(K)$, (see \cite[\S 3.1]{bgv}) which is a direct summand of the motive of the  Tate curve $\G_m^{\an}/\varpi^{\Z}$ (responsible for its $H^1$-cohomology groups). Note that this curve is (the analytification of) an elliptic curve with nodal reduction. Testing the image of this motive via classical realization functors, one can then prove the following result (see \cite[\S 3.2, \S 4.8, \S 4.9]{bgv}):
 \begin{cor}\label{quick}
         Let $\ell\neq p$ be a prime. The realizations
$$
\RigDA^{\gr}(K)\simeq\DA_N(k)\to \mcD_{\varphi,N}(\Q_\ell)
$$
$$
\RigDA^{\gr}(K)\simeq\DA_N(k)\to  \mcD_{\varphi,N}(K_0)
$$
$$
\RigDA^{\gr}(\C(\!(t)\!))\simeq\DA_N(\C)\to  (\Ind \mathcal{D}^b(\MHS^p_\Q))_N
$$
         compute  the Weil-Deligne $\ell$-adic representation attached to  the $\ell$-adic cohomology, resp. Hyodo-Kato cohomology resp. limit Hodge cohomology. 
 \end{cor}

 \begin{cor}
     If $K$ is perfectoid, the tilting equivalence $\RigDA^{\gr}(K)\simeq\RigDA^{\gr}(K^\flat)$ is compatible with the $\ell$-adic and the Hyodo-Kato realizations.
 \end{cor}
 \begin{proof}
     This is because the tilting equivalence for $\RigDA^{\gr}(K)$ can be obtained via the equivalences
     $$
     \RigDA^{\gr}(K)\simeq\DA_N(k)\simeq\RigDA^{\gr}(K^\flat).
     $$
 \end{proof}

 \subsection{The weight-monodromy conjecture}\label{sec:wm}
One of the first applications of the theory of perfectoid spaces and the almost purity theorem is Scholze's proof of the weight-monodromy conjecture for the $\ell$-adic cohomology of smooth proper varieties obtained as the intersection of smooth hypersurfaces in projective smooth toric varieties \cite{scholze}. As formulated by Deligne \cite{Deligne_hodgeI} (see \cite[Conjecture 3.9]{illusie}), the conjecture predicts in general that for any smooth projective variety $Y$ defined over a local field $K$, with residue field $k$ of characteristic $p>0$, the $i$-th $\ell$-adic cohomology groups $H^i(Y_{\overline{K}}, \Q_\ell)$ for $\ell\neq p$ have the property that the action of Frobenius on the $j$-th graded quotient with respect to the monodromy filtration $\mathrm{gr}^M_j H^i(Y_{\overline{K}}, \Q_\ell)$ is pure of weight $i+j$. This conjecture admits a $p$-adic variant, as predicted by Fontaine and Jannsen (see \cite[p. 347]{MR1012170}), where the $\ell$-adic cohomology groups are replaced by the Hyodo-Kato cohomology groups $H^i_{HK}(Y)$, which are equipped with their canonical structure of $(\varphi, N)$-modules over $W(k)[\frac{1}{p}]$.

In \cite{pWM}, we show that a part of Scholze's proof is of ``motivic nature''. As such, it can be extend almost verbatim to the case of an arbitrary realization, such as the one induced by the $p$-adic Hyodo--Kato cohomology. Recall the following definition.

\begin{dfn}
  A  (smooth) toric variety over a field $L$ is a (smooth) separated $L$-scheme $X$ containing a dense open subset $U$ isomorphic to a split torus $T=\mathbb{G}_{m,L}^r$ such that the tautological action of $T$ on $U$ by multiplication extends to $X$ (i.e., $T$ acts on $X$).  
\end{dfn} 

\begin{rmk} Every abstract normal separated toric variety can be realized as the toric variety associated to a fan $\Sigma$. More precisely, let $N$ be a free abelian group of finite rank, and write $N_\mathbb{R}$ for $N\otimes \mathbb{R}$. A fan $\Sigma$ in $N_\mathbb{R}$ is a finite collection of  strongly convex rational polyhedral cone  $\sigma\subset N_\mathbb{R}$ subject to the condition of being closed under taking faces (i.e., for each $\sigma\in \Sigma$ and each face $F$ of $\sigma$ we have $F\in \Sigma$), and such that for each $\sigma_1,\sigma_2 \in \Sigma$, we have that $\sigma_1\cap \sigma_2$ is a face of both (and so is also in $\Sigma$). 

Writing $M = \Hom(N,\Z)$, we can associate to each $\sigma\in \Sigma, \sigma \subset N_\mathbb{R}$ the affine scheme $U_\sigma = \Spec(L[\sigma^\vee \cap M])$. Since $\sigma$ is rational, it can be obtained as the cone spanned by a finite subset $S\subseteq N$. By  \cite[Theorem 1.3.12]{cox_little_schenck}, the scheme $U_\sigma$ is a non-singular $L$-scheme if the generators of $\sigma$ form a part of  a $\Z$-basis of $N$ (this is in particular independent of the field $L$). The various $U_\sigma$ for $\sigma\in \Sigma$ can be glued together to define the toric variety $X_{\Sigma, L}$. By Sumihiro's theorem \cite[Corollary 3.1.8]{cox_little_schenck}, for any  normal separated toric variety $X$, there exists a fan $\Sigma$ such that $X\simeq X_\Sigma$. In fact, $\Sigma$ can be taken as a fan in $N_\mathbb{R}$ for $N$ the character group of the torus $T$ acting on $X$.  
\end{rmk}

\begin{rmk}
A basic example of a toric variety is given by the projective space $\P^n_K$: this is obtained by the usual gluing of the affine schemes $\Spec(K[\N^n]) \simeq \A^n_K$. The reader unfamiliar with the general theory of toric varieties can safely stick to the case $X_\Sigma = \P^n$ in all statements below.  
\end{rmk}

\begin{rmk}\label{rmk:rel_frob_toric}The multiplication by $p\colon Q\to Q$ on  any monoid $Q$ induces an endomorphism on the corresponding affine toric variety $\Spec(\Z[Q])$. More generally, for any fan $\Sigma$, the multiplication by $p$ map induces a map $\varphi\colon X_\Sigma\to X_\Sigma$.
 
     If $K$ is a perfectoid field and $X_{\Sigma} = X_{\Sigma, K}$ is a toric variery over $K$ associated to a fan $\Sigma$, we can construct a perfectoid space $X_{\Sigma_\infty}^{\rm an}$ by gluing together the affinoid spaces $\Spa(K[\sigma^\vee\cap M[1/p])$. Note that this is equivalent to the weak limit $X_{\Sigma_\infty}^{\an}\sim \varprojlim_\varphi X_{\Sigma}^{\rm an}$, where $X_\Sigma^{\an}$ is the analytification of $X_\Sigma$. We will omit the superscript $\rm an$ if clear from the context. 
\end{rmk}
We can use a version of the spreading out theorem to control the cohomological motives associated to the perfectoid covers of toric varieties as follows.

\begin{cor}
\label{A=Ainf}
Let $X_{\Sigma\infty}$ be the perfectoid cover of $X_{\Sigma} = X_{\Sigma, K}$ obtained as $X_{\Sigma\infty}\sim \varprojlim_{\varphi}X_\Sigma$ where $\varphi$  is the induced by the multiplication by $p$ map as in \Cref{rmk:rel_frob_toric}. The canonical projection induces an equivalence $\mathsf{M}^{coh}(X_{\Sigma\infty})\simeq \mathsf{M}^{coh}(X_\Sigma)$.
\end{cor}
\begin{proof}
    In light of \Cref{prop:spreadingonmotives}, it suffices to show that $\varphi$ induces an equivalence on the associated (cohomological) motives. Note that $X_\Sigma$ (and the map $\varphi$ as well) has a smooth proper model $\mcX_\Sigma$ over $\mcO_K$ so that the motivic map $\varphi$ is the image under $\xi$ of the motivic map on $\mathsf{M}^{\rm coh}(\mcX_{\Sigma,k})$ induced by $\bar{\varphi}\colon  \mcX_{\Sigma,k}\to\mcX_{\Sigma,k}$. The latter is a universal homeomorphism (by looking at local coordinates, it is in fact a relative Frobenius), and hence invertible in $\DA(k)$ (see \Cref{cor:semisep}). 
\end{proof}

\begin{notn}\label{not_setting_WM}
    From now on, we  fix a smooth toric variety $X_\Sigma$ obtained from some fan $\Sigma$ of a given Euclidean space $N_\mathbb{R}$. Let $K$ be a finite extension of $\mathbb{Q}_p$, and let $\varpi$ be a uniformizer of $K$. 
     Let $Z$ be a smooth projective $K$-scheme, which is a scheme-theoretic complete intersection inside $X_{\Sigma, K}$. Write $X_\Sigma^{\rm an}$ (resp.~$Z^{\rm an}$) for the analytification of $X_{\Sigma}$ (resp. of $Z$).  
\end{notn}  

As a special case of \Cref{cor:tubes_spread}, we have the following.
\begin{cor}[{\cite[Corollary 4.5]{pWM}}]\label{tubularWM}
    Let $Z\subset X_\Sigma$ be as above.  
    There is an open neighborhood $Y\subseteq X^{\an}_{\Sigma}$ of $Z^{\an}$ for which $\mathsf{M}^{coh}(Z^{\an})\simeq \mathsf{M}^{coh}(Y)$.
\end{cor}
 Our next goal is  to revisit the proof of the weight-monodromy conjecture for $Z$, isolating the ``motivic'' part of the argument. Since the statement of the conjecture is insensitive to purely ramified extensions, we may and do replace $K$ by the completion of $K(\varpi^{1/p^{\infty}})$. In particular, we can assume that $Z$ is defined over a perfectoid field. Write $K^\flat$ for its tilt.  Similarly, write $X^\flat_\Sigma$ for the smooth toric variety defined over $K^\flat$ by the same fan used to construct $X_\Sigma$. We also write $X_{\Sigma,\infty}^{\rm an}$ (resp.~$X_{\Sigma,\infty}^{{\rm an}  \flat}$) for the perfectoid cover of $X_\Sigma$ (resp.~of $X^\flat_\Sigma$) considered in \Cref{rmk:rel_frob_toric}. 


The following approximation result is due to Scholze and is one of the essential geometric arguments for the proof. 
\begin{thm}[{\cite[Proposition 8.7, Corollary 8.8]{scholze}}]\label{tiltScholze} Let $K$ be a perfectoid field.
    Let  $Z$ be a  hypersurface in a smooth proper toric variety $X_{\Sigma}$ over $K$ and $Y$ a tubular neighborhood of $Z^{\an}$ in $X_{\Sigma}^{\an}$. There exists a  hypersurface $Z^\flat$ in $X_{\Sigma}^{\flat}$ and an open  neighborhood $Y^\flat$ of $Z^{\flat\an}$ in $X_{\Sigma}^{\flat\an}$ such that  the open $Y_\infty$ corresponds to $Y^\flat$ under the homeomorphism $|X^{\an}_{\Sigma\infty}|\simeq|X_{\Sigma\infty}^{\an\flat}|\simeq|X_{\Sigma}^{\an\flat}|$. For any chosen dense subfield of $K^\flat$, we may even suppose that the hypersurface $Z^\flat$ admits a model over it.

    More generally, if $Z = Y_1\cap \ldots \cap Y_c$ is a codimension $c$ complete intersection in $X_\Sigma$ and $Y$ is a tubular neighborhood of $Z^{\an}$ in $X_\Sigma^{\an}$, there exists a closed subvariety $Z^\flat$ in $X^\flat_\Sigma$ of codimension $c$ and an open   neighborhood $Y^\flat$ of $Z^{\flat\an}$ in $X_{\Sigma}^{\flat\an}$ such that  the open $Y_\infty$ corresponds to $Y^\flat$. Again, for any chosen dense subfield of $K^\flat$, we may  suppose that the hypersurface $Z^\flat$ admits a model over it.
\end{thm}

\begin{cons}\label{cons} Consider again the situation of \Cref{not_setting_WM}. Replacing again $K$ by  the completion of $K(\varpi^{1/p^{\infty}})$,  assume that $K$ is a perfectoid field. 
Using    \Cref{tiltScholze}, and choosing some motivic tubular neighborhood $Y$ of $Z^{\an}$ as in \Cref{tubularWM}, we obtain a diagram of adic spaces:
    $$\xymatrix{
    X_{\Sigma}^{\an} & X^{\an}_{\Sigma\infty} \ar[l]\ar@{<->}[r]& X^\flat_{\Sigma\infty}\ar[r] & X_{\Sigma}^{\flat\an}\\
    Y\ar[u]& Y_\infty\ar[u] \ar[l]\ar@{<->}[r]& Y_\infty^\flat\ar[u]\ar[r] & Y^\flat\ar[u]\\
    Z^{\an}\ar[u]&&&Z^{\flat\an}\ar[u]
    }$$
    where in the middle the arrows indicate a tilting equivalence of two perfectoid spaces over $K$ and $K^\flat$. Denote by $(-)^\sharp$ the equivalence $\RigDA(K^\flat)\simeq \RigDA(K)$  of \Cref{cor:tilting}.  The diagram above induces the following diagram of \emph{motives} over $K$:
      $$\xymatrix{
    \mathsf{M}^{\rm coh}(X_{\Sigma}^{\an}) \ar[r]^{\sim}\ar[d]& \mathsf{M}^{\rm coh}(X^{\an}_{\Sigma\infty})\ar@{-}[r]^{\sim}\ar[d]& \mathsf{M}^{\rm coh}(X^\flat_\infty)^\sharp \ar[d]&\mathsf{M}^{\rm coh}( X_{\Sigma}^{\flat\an})^\sharp\ar[l]_{\sim}\ar[d]\\
    \mathsf{M}^{\rm coh}(Y)\ar[d]^{\sim}\ar[r]^{\bigoplus}& \mathsf{M}^{\rm coh}(Y_\infty)\ar@{-}[r]^{\sim}& \mathsf{M}^{\rm coh}(Y_\infty^\flat)^\sharp &\mathsf{M}^{\rm coh}( Y^\flat)^\sharp\ar[d]\ar[l]_{\sim}\\
   \mathsf{M}^{\rm coh}( Z^{\an})&&&\mathsf{M}^{\rm coh}(Z^{\flat\an})^\sharp
    }$$
    where the symbol $\oplus$ denotes an inclusion of a direct factor (by \Cref{dirfact}, but this is not important), the equivalences in the middle follow from \Cref{tilt}, the one in the bottom left follows from the choice of $Y$ (see \Cref{tubularWM}), the one on the top left from \Cref{A=Ainf} and the ones on the right from \Cref{cor:semisep}.

    In particular, we have a map
    $
    \mathsf{M}^{\rm coh}(Z^{\an})\to \mathsf{M}^{\rm coh}(Z^{\flat\an})^\sharp$.  We may even use \cite[Theorem 4.1]{dJong} to find a smooth alteration $T^\flat$ of $Z^\flat$ therefore giving rise to a map
    $
    \mathsf{M}^{\rm coh}(Z^{\an})\to \mathsf{M}^{\rm coh}(T^{\flat\an})^\sharp.$
    We also note that this  map is compatible with the diagonal maps $Z\to Z\times Z$ resp. $T^\flat\to T^\flat\times T^\flat$ (as it is given by a composition of geometric maps, tilting equivalences and inverses thereof). 
\end{cons}

We are going to show that $\mathsf{M}^{\rm coh}(Z^{\rm an})$ is a \emph{direct summand} of $\mathsf{M}(T^\flat)^\sharp$. This was observed by Scholze at the level of $\ell$-adic cohomology groups, but in fact the proof only relies on a form of Poincaré duality, hence on motivic results, as we now explain. 
\begin{cons}
Let $X$ be an algebraic [resp. analytic] variety over $K$. The associated cohomological motive $M\colonequals p_*\one=\mathsf{M}^{\rm coh}(X)\in\DA(K)$ [resp. in $\RigDA(K)$] inherits a ring structure. If $X$ is smooth, proper of dimension $d$, then there is a natural duality pairing
$$
M\otimes M\to M\to \one(-d)[-2d]
$$
which exhibits $M$ as its self-dual, up to twist and shift: $M\simeq\uhom(M,\one(-d)[-2d])$.

Assume now $M'$ is another ring object, equipped with a similar duality pairing giving rise to $M'\simeq\uhom(M',\one(-d)[-2d])$, for the same $d$ as before. Any ring map $f\colon M\to M'$ then gives also rise to a dual map
$$
f^\vee\colon M'\simeq \uhom(M',\one(-d)[-2d])\xto{f^*}\uhom(M,\one(-d)[-2d])\simeq M.
$$ 
Note that the composition $f^\vee\circ f\colon M\to M$ corresponds to a pairing $M\otimes M\to \one(-d)[-2d]$ obtained as follows:
$$
\xymatrix{
M\otimes M\ar[d]^{f\otimes \id}\ar[r]^{\mu} & M\ar[dd]^{f} \ar[r] & \one(-d)[-2d]\ar@{.>}[dd]^{\sim}\\
M'\otimes M\ar[d]^{\id\otimes f}\\
M'\otimes M'\ar[r]^{\mu} & M' \ar[r] & \one(-d)[-2d]
}
$$
In case there is an invertible automorphism (the dotted arrow above) of $\one(-d)[-2d]$ making the  diagram commute, then we can deduce that the map $f^\vee\circ f$ is equivalent to the map corresponding to the duality pairing (i.e. the identity map of $M$) and hence that $M$ is a direct summand of $M'$. Luckily, there is an easy criterion to fill in the missing edge, as the next proposition shows.
\end{cons}

\begin{prop}\label{rmk:split}
    Let $X$ be a connected smooth and proper algebraic variety over $K$ of dimension $d$, with semi-stable reduction. Then $\Hom(\mathsf{M}^{\rm coh}(X^{\an}),\one(-d)[-2d])\simeq\Q$. In particular, all non-zero maps $\mathsf{M}^{\rm coh}(X)\to\one(-d)[-2d]$ are equivalent with each other, up to a scalar multiplication.
\end{prop}
\begin{proof}
By duality, we may alternatively prove that $\Hom(\one(d)[2d],\mathsf{M}(X))\simeq \Q$. This group is, by adjunction, the same as $\Hom(\one(d)[2d],\chi \mathsf{M}(X))$. The object on the left is pure with respect to the Chow-weight structure on $\RigDA^{\gr}(K)$ (see \cite[\S 4.3]{bgv}). For the second, there is an explicit description of its weight complex  in terms of the irreducible components $D_i$ of its special fiber as shown in \cite[Example 4.72]{bgv}. More precisely (we exhibit the complex starting from degree 2, we let $D_I$ be $\bigcap_{i\in I} D_i$ and omit the notation $\mathsf{M}(-)$):
$$w(\chi \mathsf{M}(X))= \cdots\to \bigoplus D_{ij}(-1)[-2]\to \bigoplus D_i(-1)[-2]\to \bigoplus D_i\to \bigoplus D_{ij}(1)[2]\to \cdots
$$
We remark that $\Hom(\one(d)[2d],\bigoplus D_i)\simeq \bigoplus \Q$ and  $\Hom(\one(d)[2d],\bigoplus D_{ij}(1)[2])\simeq \bigoplus \Q$. The terms in positive degree $\Hom(\one(d)[2d],\bigoplus D_{I}(1)[2])$ are on the other hand equal to $0$ for $|I|> 2$. This follows from the identification (as a consequence of e.g.~\cite[Corollary 2]{voev_chow})
\begin{align*}\Hom_{\DA(k)}(\one(d)[2d], \mathsf{M}(D_{I})(1)[2]) &\simeq \Hom_{\DA(k)}(\mathsf{M}(D_{I}), \one(1-d+r)[2(1-d+r)])  \\ &\simeq \mathrm{CH}^{1-d+r}(D_I)_\Q\end{align*}
where $r = \dim(D_I)=d-|I|+1$. The exponent $1-d+r$ is negative for $|I|>2$, so the term is zero. 
 We deduce that $\Hom(\one(d)[2d],\mathsf{M}(X))$ is the kernel of the natural ``difference'' map $\bigoplus_{i}\Q\to \bigoplus_{i\neq j}\Q$ which is given by the diagonal embedding $\Q\to \bigoplus_{i}\Q$.
\end{proof}

\begin{cor}\label{original}
    Let $Z$ be a smooth scheme-theoretic complete intersection in a smooth projective toric variety $X_\Sigma$. The weight-monodromy conjecture holds for the Hyodo--Kato cohomology of $Z$. 
\end{cor}

\begin{proof}
As before, we may and do replace $K$ with a perfectoid extension, 
and even assume that $Z$ has semi-stable reduction (as we may take any finite extension of $K$).  We let $T^\flat$ be as in \Cref{cons}. We may and do suppose that it is defined over some finite extension $L$ of the field $k(\!(p^\flat)\!)\subseteq K^\flat$ and that $\mathsf{M}^{\rm coh}(T^\flat)\in\RigDA^{\gr}(L)$.

    We then consider the Hyodo--Kato realization functor
    $$
    \RigDA^{\gr}(K)\to\mcD_{\varphi,N}(W(k)[1/p])
    $$
    and observe that:
    \begin{enumerate}
        \item The composite realization 
        $$
         \RigDA^{\gr}(L)\to \RigDA^{\gr}(K^\flat){\simeq}\RigDA^{\gr}(K)\to\mcD_{\varphi,N}(W(k)[1/p])
        $$
        must agree with the Hyodo--Kato realization functor defined by Lazda--Pal \cite[Theorem 2.52]{lazda-pal} by means of \Cref{quick}. 
                \item The map of motives
        $$
        M\colonequals \mathsf{M}^{\rm coh}(Z)\to M'\colonequals \mathsf{M}^{\rm coh}(T^{\flat})^\sharp
        $$
         is a map of ring objects  (since the multiplication is induced by the diagonal map, see the end of \Cref{cons}) and the map $M\to M'\to\one(-d)[-2d]$ fits in the following diagram
        \begin{equation}
         \label{diag}
        \xymatrix{
        \mathsf{M}^{\rm coh}(X^{\an}_{\Sigma})\ar[d]^{\sim}\ar[r] & M\ar[d]\ar[dr]\\
        \mathsf{M}^{\rm coh}(X^{\flat\an}_{\Sigma})^\sharp\ar[r] & M'\ar[r] & \one(-d)[-2d]
        }
        \end{equation}
       \item  
       The lower horizontal map  is the (untilt of the) analytification of the algebraic map
        $$
        \mathsf{M}^{\rm coh}(X_\Sigma^\flat)\to \mathsf{M}^{\rm coh}(T^\flat)\to\one(-d)[-2d]
        $$
        in $\DA(L)$. All the motives above are dualizable,  and the analytification functor on motives is monoidal. We then deduce a diagram (where we denote by $\{n\}$ the operation $(n)[2n]$, for brevity):
        $$\xymatrix{
        \Hom_{\DA(L)}(\one\{-d\},\mathsf{M}^{\rm coh}(X^\flat_\Sigma))\simeq  CH^d(X^\flat_\Sigma)_\Q \ar[r]\ar[d]^{\iota^!}&    \Hom_{\RigDA(L)}(\one\{-d\},\mathsf{M}^{\rm coh}(X^{\an\flat}_\Sigma)) \ar[d] \\
        \Hom_{\DA(L)}(\one\{-d\},\mathsf{M}^{\rm coh}(T^\flat))\simeq  CH^d(T^\flat)_\Q \ar[r]\ar[d]& \Hom_{\RigDA(L)}(\one\{-d\},\mathsf{M}^{\rm coh}(T^{\an\flat}))\ar[d]\\
        \Hom_{\DA(L)}(\one\{-d\},\one\{-d\})\ar[d]^{\sim}\ar[r] &\Hom_{\RigDA(L)}(\one\{-d\},\one\{-d\})\ar[d]^\sim\\
       \Q \ar[r]^\sim& \Q
      }  $$
      where $\iota$ is the inclusion $T^\flat \subset X_\Sigma^\flat$. 
         Note that the left column is a non-zero map. Indeed, the composition
         \[ \mathrm{CH}^d(X_\Sigma^\flat) \xrightarrow{\iota^!} \mathrm{CH}^d(T^\flat) \xrightarrow{\deg} \Z\]
        is always non-zero (it suffices to notice that $\deg (h^!(c_1(\mathcal{L})^d))\neq 0$ where $\mathcal{L}$ is an ample line bundle on $X_\Sigma^\flat$   by \cite[Lemma 12.1]{fulton}).
        
        As a consequence, we deduce that  the composite map in the right column  is non-zero as well. In particular, the lower horizontal map of the diagram \ref{diag} is non-zero, hence so is the diagonal map.
        
%
        
    \end{enumerate}
    We then deduce by  \Cref{rmk:split} that $\mathsf{M}^{\rm coh}(X)$ is a direct summand of $\mathsf{M}^{\rm coh}(T^\flat)^\sharp$, and so are its Hyodo--Kato cohomology groups (that are $(\varphi,N)$-modules).  We can then conclude that $H_{HK}^*(X)$ satisfies the weight-monodromy conjecture, since $H_{HK}^*(T^\flat)$ does: this is a result by Crew \cite{Crew} and Lazda--Pal \cite[Theorem 5.58]{lazda-pal}, that builds on a proof of Deligne \cite{weil2}. 
\end{proof}

\begin{rmk}
    In the previous proof, one can test the ``zero-ness'' of the map $M\to \one(-d)[-2d]$ under any given realization functor. In particular, using the notation above, if one shows that the induced map $H^{2d}(X)\to H^{2d}(M')\simeq \one(d)$ is non-zero for some cohomological realization, then one can deduce that $M=\mathsf{M}^{\rm coh}(X)$ is a direct summand of $M'$. This is the approach taken in \cite{scholze} (and in \cite{pWM}).
\end{rmk}

 \subsection{On the Clemens-Schmid sequence} The motivic nearby cycle can also be used to give a uniform treatment (in the $\ell$-adic, $p$-adic and in the complex analytic setting) of the so-called local invariant cycle theorem. 
 
 Let us begin by casting the problem in the complex analytic case. 
   Let $f\colon X\to S$ be a proper morphism between smooth complex analytic spaces, where $S$ has dimension one. Localizing around an analytic neighborhood of $0\in S(\C)$ and passing to a finite cover, we may replace $S$ by an open disc $\Delta$, suppose that the central fiber $Y=f^{-1}(0)$ is a normal crossing divisor in $X$, and that $f$ is smooth away from zero. Write $X_t$ for the (smooth) fiber of $f$ at a point $t\neq0$. Since the restriction of $f$ to $\Delta^*$ is topologically a locally trivial bundle, a generator of the fundamental group $\pi_1(\Delta^*)\simeq \Z$ acts on the vector space $H^q(X_t, \Q)$ by an operator $N$, which is in fact nilpotent. As such, it induces a canonical filtration $M$, for which $N^r$ restricts to an isomorphism $\gr^M_{r} H^q(X_t, \Q)\to \gr^M_{-r}H^q(X_t, \Q)$ The morphism $N$ can be promoted to a morphism of mixed Hodge structures when the vector space $H^q(X_t, \Q)$ is  equipped with Steenbrink's limit Hodge structure $H^n_{\lim}(X, \Q(m))$ from \cite{steenbrink}. 
   The kernel of the monodromy operator captures some cohomological information of the singular space $Y$.  More precisely, one has the following result:
  \begin{thm}[Local invariant cycle theorem {\cite[3.7]{clemens}}]\label{thm:local_invariant_cycle}Assume that $X$ is a  K\"ahler manifold. The sequence
  \begin{equation}\label{eq:loc_cyc}H^n(Y, \Q(m)) \xrightarrow{\mathrm{sp}} H^n_{\lim}(X, \Q(m)) \xrightarrow{N} H^n_{\lim}(X, \Q(m-1))\end{equation}
    where $\mathrm{sp}$ is the canonical specialization morphism, is exact in the category of mixed Hodge structures. 
  \end{thm}
The specialization morphism arises as follows. Let $\overline{X}^*$ be the base change of $X^* = X\times_\Delta \Delta^*$ to a universal cover $\overline{\Delta^*} \to \Delta^*$. Then for $t\neq 0$ we have $H^*(\overline{X}^*, \Q) \simeq H^*(X_t, \Q)$ as $\Q$-vector spaces, and there is a natural morphism  $H^n(Y, \Q(m)) \to H^n(\overline{X}^*, \Q(m))$ by localization. The main content of the local invariant cycle theorem is that this is indeed a morphism of mixed Hodge structures when $H^*(X_t, \Q)$ is equipped with the limit Hodge structure $H^n_{\lim}(X, \Q(m))$, and that the kernel of the monodromy map $N$ has the correct weights. 

Using Poincaré duality one can  in fact  complete the exact sequence \eqref{eq:loc_cyc} to the left and to the right and obtain a long exact sequence:
\begin{thm}[Clemens-Schmid sequence {\cite{morrison}, \cite{clemens}}]\label{thm:CS}Assume that $X$ is a  K\"ahler manifold. Then there are (two) long exact sequences 
    \[
    \begin{tikzcd}
       \dots \to  H^n(Y, \Q(m)) \arrow[r, "\mathrm{sp}"] & H^n_{\lim}(X, \Q(m)) \arrow[r, "N"]
             \ar[draw=none]{d}[name=X, anchor=center]{}
    & H^n_{\lim}(X, \Q(m-1)) \ar[rounded corners,
            to path={ -- ([xshift=2ex]\tikztostart.east)
                      |- (X.center) \tikztonodes
                      -| ([xshift=-2ex]\tikztotarget.west)
                      -- (\tikztotarget)}]{dll}[at end]{} \\      
  H_{2d-n}(Y, \Q(m-d-1)) \rar & H^{n+2}(Y, \Q(m-1))  \arrow[r, "{\mathrm{sp}}"] & H^{n+2}_{\lim}(X, \Q(m-1))\to \dots
    \end{tikzcd}
    \]
 of mixed Hodge structures. 
    \end{thm}
The local invariant cycle theorem is then part of   the exactness of Clemens-Schmid sequence. Note that by duality the exactness at the spot $H^n_{\lim}(X, \Q(m))$ is essentially enough to prove exactness at the remaining terms (see the remark at \cite[p.~229]{clemens}). In turn, this is a consequence of the complex analytic version of the weight-monodromy conjecture. 
\begin{thm}[Monodromy-Weight, Hodge version {\cite[Th\'eor\`eme 5.2]{g-na}}]
    In the above assumptions, for each $r\geq 0$ and each $n$, the morphism $N^r\colon H^n_{\lim}(X, \Q)\to H^{n}_{\lim}(X, \Q(-r))$ induces an isomorphism  of pure Hodge structures
    \[
    \mathrm{gr}^W_{n+r} H^n_{\lim}(X, \Q)\xrightarrow{\simeq}  \mathrm{gr}^W_{n-r} H^n_{\lim}(X, \Q)
    \]
    where $W$ is the weight filtration of the Hodge structure $H^n_{\lim}(X, \Q)$. Equivalently, the weight and the monodromy filtrations agree (up to a shift by $n$).
\end{thm}
For the proof of the exactness of the Clemens-Schmid sequence, it is enough to know that the kernel of the monodromy $N$ on $H^i_{\lim}(X)$ has weights $\leq i$. But this is precisely guaranteed by the concidence of the weight filtration and of the monodromy filtration on $H^i_{\lim}(X)$. 

  \begin{rmk}\Cref{thm:local_invariant_cycle} is due to Deligne (see \cite{Deligne_hodgeI}) in the algebraic setting, and to Clemens in the stated generality. The most general approach using the theory of mixed Hodge modules is due to Saito \cite{saito-mdhp} (see \cite{illusie} for an extensive discussion). Note that the assumptions are optimal: as remarked by Nakkajima \cite{Nakka} (see also the discussion in the original paper of Clemens \cite{clemens}), the statement is false for arbitrary proper families of complex analytic spaces. From the motivic point of view, this implies that a generalization of the local invariant cycle theorem cannot hold for arbitrary objects of $\RigDA(\C(\!(t)\!))^{\rm gr}$.
\end{rmk}
\begin{rmk} Following the approach of Deligne in the proof of the $\ell$-adic weight-monodromy conjecture (see the proof of \cite[Théorème 1.8.4]{weil2}), having an estimate for the weights of the kernel of the monodromy operator on $H^i$, together with a product formula, is sufficent to establish the full conjecture. For this reason, we expect that the exactness of the Clemens-Schmid sequence for the variety $X$ and its self-products is in fact equivalent to the weight-monodromy conjecture.
\end{rmk}
From our perspective, all of the above arises more naturally as realization of the monodromy action on the motivic nearby cycle. To explain this, let us consider the following general setting. 
\begin{cons} Let $K$ be a complete non-archimedean field with a perfect residue field and a value group of (rational) rank $1$. 
			Let $k$ be its residue field and denote by $\mcO_K$ the valuation ring. Let us now write $\DA(\mcO_K)^{\rm gr}$ [resp. $\DA(K)^{\rm gr}$] for the full subcategory of $\DA(\mcO_K)$ [resp. $\DA(K)^{\rm gr}$] defined as the inverse image of $\RigDA(K)^{\rm gr}$ along the composite
                        \[\DA(\mcO_K)\xrightarrow{j^*} \DA(K) \xrightarrow{\An^*} \RigDA(K). \]
                        [resp. $\An^*$]. 
This category in particular contains the motive of any regular scheme $\mathcal{X}$, flat over $\mcO_K$ and with strict semi-stable reduction over $k$.  We then have:
           \[ \begin{tikzcd}
                & \DA^{\gr}(\mcO_K)\arrow[d, "j^*"
]\\
& \DA^{\gr}(K) \arrow[d, "\An^*"] \arrow[rdd, bend left=35, "\chi^{\rm alg}"] \arrow[ldd, bend right=35, "\Psi^{\alg}"']\\
             & \RigDA^{\gr}(K)\arrow[d, "\sim"
]\arrow[rd, bend left=15, "\chi"] \arrow[ld, bend right=15, "\Psi"'] \\
            \DA(k) & \DA_N(k)\arrow[l, "\pi"'] \arrow[r, "\fib N
"] &\DA(k)
            \end{tikzcd}\]
            where $\chi^{\rm alg}$ is the motivic algebraic specialization functor $\iota^*j_*$, and $\chi$ is the motivic analytic specialization functor, right adjoint to the generic fiber functor $\xi$.  The commutations on the left are by definition, while on the right they follow from  \cite[Corollary~3.8.18 and Theorem 3.8.19]{agv}, and \Cref{rmk:identifadj}.
            

                        Combining the localization sequence of \Cref{item:loc}-(v) with the previous discussion we obtain the following commutative diagram of natural transformation of functors from $\DA(\mcO_K)^{\rm gr}$ to $\DA(k)$

\begin{equation}\label{eq:calotte}
\begin{tikzcd}
               \iota^*
			\arrow[dd, dashed, "+1"]
			\ar[dr]
			&&
                \iota^!
                \ar[ll]
\\
& \chi^{\mathrm{alg}} j^*\ar[ur,"+1"]\ar[dl,"+1"]
\\
\Psi \An^* j^*[-1]\ar[rr,"N"]&&\Psi \An^*j^*(-1)[-1]\ar[ul]\ar[uu,"+1",dashed]
		\end{tikzcd}
\end{equation}
 \end{cons}We can  compose the square above with any realization functor $R\Gamma\colon\DA(k)\to \catD$   to obtain a Clemens--Schimd style complex. More precisely, we have the following result. 

\begin{cor}[{\cite[Remark 4.68, Corollary 4.70]{bgv}}]Let $R\Gamma\colon\DA(k)\to \catD$ be a monoidal functor of presentable stable \icats, and assume that $\catD$ is endowed with a t-structure compatible with tensor products and with heart~$\mcA$. Let $\widehat{R\Gamma}$ be the canonical extension of \eqref{eq:def_Fhat}, following \Cref{thm:DA_N_is_RigDA}. 

                             Let  $\mcX$ be a flat, proper, regular scheme of essential relative dimension $d$ over $\mcO_K$ which is generically smooth and with polystable  special fiber $Y$ over $k$ (so that $\An^*j^*\mathsf{M}(\mcX)= \mathsf{M}(\mcX_K^{\An}) \in \RigDA(K)^{\rm gr}$ by \Cref{prop:semiss_is_gr}). %
						Then there are (two) chain complexes in $\mcA$:
						\[
						\cdots \to \Hm^*(Y)\to \widehat{\Hm}{}^*(\mcX_K^{\An})\xto{N} \widehat{\Hm}{}^*(\mcX_K^{\An})(-1) \to \Hm_{2d-*}(Y)(-d-1) \to  \Hm^{*+2}(Y)\to\cdots 
						\]
					where we let $\Hm_*(M)$ resp. $\widehat{\Hm}_{*}(M)$   be the objects $\Hm_*R\Gamma(M)$ resp. $\Hm_*\pi\widehat{R\Gamma}(M)\cong \Hm_*R\Gamma(\Psi M)$  and we let  $\Hm^{*}(M)$ resp. $\widehat{\Hm}{}^{*}(M)$ be their duals.
\end{cor}
\begin{proof}The Corollary follows immediately by applying the realization functor to the square \eqref{eq:calotte}, together with absolute purity. 
\end{proof}
When $\widehat{R\Gamma}$ is the canonical extension of the Hodge realization $R\Gamma_{\rm Hodge}\colon \DA(\C)\to \Ind\mcD^b(\MHS^p_\Q)$ 
 of \Cref{ex:Betti} 
     we obtain exactly the sequence in \Cref{thm:CS}. 

\begin{rmk}When $\widehat{R\Gamma}$ is the Hyodo-Kato realization functor, we obtain in particular that for every $\mcX/\mcO_K$ as above a \emph{complex} of $\varphi$-modules
\[\ldots\to  H^{n}_{\rig}(Y) \xrightarrow{{\rm sp}} H^n_{HK}(\mcX_K^{\An}) \xrightarrow{N} H^n_{HK}(\mcX_K^{\An})(-1) \to \ldots  \]
that is exact when the $p$-adic weight-monodromy conjecture holds. Exactness at the spot $H^n_{HK}(\mcX_K^{\An})$ is the $p$-adic analogue of the local invariant cycle theorem. The existence of such sequence (as well as the exactness statement) was originally conjectured by Chiarellotto \cite{chiar-duke} in the semi-stable setting, and then conjecturally extended to the analogue of the Clemens-Schmid sequence by Flach and Morin \cite{flach-morin}. To our knowledge, not even the morphisms appearing in the 
$p$-adic Clemens-Schmid complex had been explicitly defined in this generality prior to \cite{bgv}.
\end{rmk}
\begin{rmk}
In more classical terms, for \emph{proper} $\mcX$, the identifications of \cite{HK94} allow to describe the map $\mathrm{sp}$ as follows. First, there is a morphism (constructed by Chiarellotto in \cite{chiar-duke})
\[
H^{n}_{\rig}(Y/K_0) \to H^n((Y,N)/(W(k),\N))\otimes K_0
\]
where $(Y,N)$ is the proper semistable log scheme over the log point $k^0$ obtained by equipping the special fiber of $\mcX$ with its canonical log structure, and $H^n((Y,N)/(W(k),\N))$ is the log crystalline cohomology group relative to the base $(W(k), L)$, with $1\mapsto 0$.
The Hyodo-Kato isomorphism \cite[Proposition 4.13]{HK94} further gives an identification
\[
H^n((Y,N)/(W(k),\N))\otimes K_0 \simeq H^n((Y,N)/(W(k)^\times)\otimes K_0
\]
where the right-hand side denotes the crystalline cohomology of $(Y,N)$ computed relative to the standard log structure on $W(k)$ given by $1\mapsto p$. This is  in turn  isomorphic to the de Rham cohomology of the (analytic) generic fiber $H^n_{\rm dR}(\mcX_K^{\An})$ by \cite[Theorem 5.1]{HK94} (in this case clearly agrees with algebraic de Rham cohomology, since $\mcX_K$ is smooth and proper over $K$). Our method completely bypasses the reference to the log crystalline cohomology groups (relative to different log bases), and constructs directly the map $\mathrm{sp}$ between $H^n_{\rm rig}(Y)$ and $H^n_{HK}(\mcX_K^{\An})$. 

In log-terms, using the comparison isomorphisms above,  we can interpret the Clemens--Schmid complex  as connecting the   rigid cohomology  and the log-crystalline cohomology (after inverting $p$) of the \emph{singular} special fiber $Y$.
\end{rmk}


\bibliography{pisabib}
\bibliographystyle{siam}

 \end{document}